\documentclass{amsart}

\newcommand{\A}{\mathcal{A}}

\usepackage{amsmath}
\usepackage{amsthm}
\usepackage{amsfonts}
\usepackage{amssymb}
\usepackage{mathrsfs}
\usepackage{amscd}
\usepackage{url}

\newcommand{\re}{\mathbb{R}}
\newcommand{\cpx}{\mathbb{C}}

\newcommand{\N}{\mathbb{N}}

\renewcommand{\P}{\mathbb{P}}

\newcommand{\half}{\frac{1}{2}}
\newcommand{\lmd}{\lambda}

\newcommand{\eps}{\epsilon}

\newcommand{\dt}{\delta}

\def\af{\alpha}
\def\bt{\beta}

\def\rank{\mbox{rank}}

\newcommand{\sig}{\sigma}
\newcommand{\Sig}{\Sigma}

\newcommand{\reff}[1]{(\ref{#1})}

\newcommand{\mc}[1]{\mathcal{#1}}

\newcommand{\supp}[1]{\mbox{supp}(#1)}

\newcommand{\bdes}{\begin{description}}
\newcommand{\edes}{\end{description}}

\newcommand{\bal}{\begin{align}}
\newcommand{\eal}{\end{align}}

\newcommand{\bnum}{\begin{enumerate}}
\newcommand{\enum}{\end{enumerate}}

\newcommand{\bit}{\begin{itemize}}
\newcommand{\eit}{\end{itemize}}

\newcommand{\bea}{\begin{eqnarray}}
\newcommand{\eea}{\end{eqnarray}}
\newcommand{\be}{\begin{equation}}
\newcommand{\ee}{\end{equation}}

\newcommand{\baray}{\begin{array}}
\newcommand{\earay}{\end{array}}

\newcommand{\bsry}{\begin{subarray}}
\newcommand{\esry}{\end{subarray}}

\newcommand{\bca}{\begin{cases}}
\newcommand{\eca}{\end{cases}}

\newcommand{\bcen}{\begin{center}}
\newcommand{\ecen}{\end{center}}

\newcommand{\bbm}{\begin{bmatrix}}
\newcommand{\ebm}{\end{bmatrix}}

\newcommand{\bmx}{\begin{matrix}}
\newcommand{\emx}{\end{matrix}}

\newcommand{\bpm}{\begin{pmatrix}}
\newcommand{\epm}{\end{pmatrix}}

\newcommand{\btab}{\begin{tabular}}
\newcommand{\etab}{\end{tabular}}

\newtheorem{theorem}{Theorem}[section]
\newtheorem{pro}[theorem]{Proposition}
\newtheorem{prop}[theorem]{Proposition}
\newtheorem{lem}[theorem]{Lemma}

\newtheorem{cor}[theorem]{Corollary}

\newtheorem{ass}[theorem]{Assumption}

\theoremstyle{definition}

\newtheorem{exm}[theorem]{Example}
\newtheorem{alg}[theorem]{Algorithm}

\newtheorem{remark}[theorem]{Remark}

\begin{document}

\title[The $\mathcal{A}$-truncated $K$-moment problem]
{The $\mathcal{A}$-truncated $K$-moment problem}

\author[Jiawang Nie]{Jiawang Nie}
\address{Department of Mathematics, University of California San Diego,
9500 Gilman Drive, La Jolla, CA 92093, USA.}

\email{njw@math.ucsd.edu}

\begin{abstract}
Let  $\A \subseteq \N^n$ be a finite set,
and $K\subseteq \re^n$ be a compact semialgebraic set.
An {\it $\A$-truncated multisequence} ($\A$-tms) is a vector
$y=(y_{\af})$ indexed by elements in $\A$.
The $\A$-truncated $K$-moment problem ($\A$-TKMP) concerns whether or not
a given $\A$-tms $y$ admits a $K$-measure $\mu$,
i.e., $\mu$ is a nonnegative Borel measure supported in $K$
such that $y_\af = \int_K x^\af \mathtt{d}\mu$ for all $\af \in \A$.
This paper proposes a numerical algorithm for solving $\A$-TKMPs.
It aims at finding a flat extension of $y$
by solving a hierarchy of semidefinite relaxations
$\{(\mathtt{SDR})_k\}_{k=1}^\infty$ for a moment optimization problem,
whose objective $R$ is generated in a certain randomized way.
If $y$ admits no $K$-measures and $\re[x]_{\A}$ is $K$-full
(there exists $p = \sum_{\af\in \A} p_\af x^\af$ that is positive on $K$),
then $(\mathtt{SDR})_k$ is infeasible for all $k$ big enough,
which gives a certificate for the nonexistence of representing measures.
If $y$ admits a $K$-measure, then for almost all generated $R$,
this algorithm has the following properties:
i) we can asymptotically get a flat extension of $y$
by solving the hierarchy $\{(\mathtt{SDR})_k\}_{k=1}^\infty$;
ii) under a general condition that is almost sufficient and necessary,
we can get a flat extension of $y$ by solving $(\mathtt{SDR})_k$ for some $k$;
iii) the obtained flat extensions admit a $r$-atomic $K$-measure with $r\leq |\A|$.
The decomposition problems for completely positive matrices
and sums of even powers of real linear forms,
and the standard truncated $K$-moment problems,
are special cases of $\A$-TKMPs.
They can be solved numerically by this algorithm.
\end{abstract}

\keywords{$\A$-truncated multisequence,
$\A$-truncated $K$-moment problem,
completely positive matrices, flat extension, moment matrix,
localizing matrix, representing measure,
semidefinite program, sums of even powers}

\subjclass{44A60, 47A57, 90C22, 90C90}

\maketitle

\section{Introduction}

Let $\A \subseteq \N^n$ be a finite set ($\N$ is the set of nonnegative integers).
An {\it $\A$-truncated multisequence} ($\A$-tms) is a vector $y:=(y_{\af})_{\af \in \A}$
in $\re^{\A}$ (the space of real vectors indexed by elements in $\A$).
For $\af :=(\af_1, \ldots,\af_n) \in \N^n$, denote $|\af|:=\af_1+\cdots+\af_n$.
The degree of $\A$ is
$
\deg(\A):=\max\{ |\af|: \, \af \in \A \}.
$
Let $K$ be the semialgebraic set
\be  \label{def:K}
K := \left\{x\in \re^n:\, h(x)  = 0,  g(x) \geq 0 \right\}
\ee
defined by two tuples of polynomials
$h:=(h_1,\ldots, h_{m_1})$ and $g:=(g_1,\ldots, g_{m_2})$.
A nonnegative Borel measure $\mu$ on $\re^n$ is called a {\it $K$-measure}
if its support, denoted by $\supp{\mu}$, is contained in $K$.
For $x :=(x_1, \ldots,x_n) \in \re^n$ and
$\af :=(\af_1, \ldots,\af_n) \in \N^n$,
denote $x^\af := x_1^{\af_1}\cdots x_n^{\af_n}$.
The integral $\int_K x^\af \mathtt{d}\mu$, if it exists,
is called the $\af$-th moment of a $K$-measure $\mu$.
An $\A$-tms $y$ is said to admit the measure $\mu$ if
for all $\af \in \A$, the moment
$\int_K x^\af \mathtt{d}\mu$ exists and is equal to $y_\af$.
%
%
Such $\mu$ is called a {\it $K$-representing} measure for $y$.
Let $meas(y,K)$ denote the set of all $K$-measures admitted by $y$. Denote
\[
\mathscr{R}_{\A}(K) := \{ y \in \re^{\A}: \, meas(y,K) \ne \emptyset \}.
\]
The $\A$-truncated $K$-moment problem ($\A$-TKMP) concerns whether
a given $\A$-tms $y$ admits a $K$-measure or not.
If it does not, can we get a certificate for that?
If it does, how can we obtain a $K$-representing measure?
Preferably, we are often interested in finitely atomic measures.
(A measure is {\it finitely atomic} if its support is a finite set,
and is {\it $r$-atomic} if its support consists of at most $r$ distinct points.)
This paper presents a numerical algorithm for solving $\A$-TKMPs.
The $\A$-truncated moment problems appear frequently in applications. For instance,
in sparse polynomial optimization, the variables in its semidefinite relaxations
(cf.~Lasserre~\cite{LasSpar}) are $\A$-tms'
whose $\A$ depends on the sparsity patterns.
For such $\A$, the truncated moment problem was studied in
Laurent and Mourrain\cite{LauMou09}.

\subsection{Two special cases}

Many hard computational problems can be formulated as
$\A$-TKMPs with appropriate $\A$ and $K$.
Here we list two of them.

The first one is the decomposition problem for {\it completely positive}
matrices (cf.~\cite{BerSM03}).
A symmetric $n\times n$ matrix $C$ is called
{\it completely positive} if there exist vectors
$u_1,\ldots, u_r \in \re_+^n$
(the nonnegative orthant of $\re^n$) such that
\[
C = u_1u_1^T+\cdots+u_ru_r^T.
\]
The above is called a {\it CP-decomposition} of $C$, if it exists.
How can we determine whether a matrix $C$ is completely positive or not?
If it is not, can we get a certificate for that?
If it is, how can we get a CP-decomposition for $C$?
As we will show in Section~\ref{sec:appl},
this problem can be formulated as an $\A$-TKMP with
\[
\A =\{ \af\in \N^n:\, |\af|=2 \}, \quad
K=\{ x\in \re_+^n:\, x_1+\cdots+x_n = 1 \}.
\]

The second one is the decomposition problem for
{\it sums of even powers} (SOEP) of real linear forms (cf.~\cite{Rez92}).
(A form is a homogeneous polynomial.)
A form $f$ of an even degree $m$ is SOEP if
there exist real linear forms $L_1,\ldots, L_r$ such that
\[
f = L_1^m+\cdots+L_r^m.
\]
The above is called an {\it SOEP-decomposition} of $f$, if it exists.
How can we determine whether a form $f$ is SOEP or not?
If it is not, can we get a certificate for that?
If it is, how can we get an SOEP-decomposition for $f$?
As we will show in Section~\ref{sec:appl},
this problem can also be formulated as an $\A$-TKMP with
\[
\A = \{ \af\in \N^n:\, |\af|=m \}, \quad
K = \{ x \in \re^n:\, x^Tx = 1, x_1+\cdots+x_n \geq 0 \}.
\]

It is typically quite difficult to detect the existence of
CP/SOEP-decompositions, or to compute them when they exist.
In the prior existing work, there are no much efficient numerical methods
for solving such decomposition problems (except some special cases),
in the author's best knowledge. In this paper, we show that they
can be solved numerically as special cases of $\A$-TKMPs.

\subsection{Standard truncated $K$-moment problems}

Denote $\N_d^n:=\{\af \in \N^n: |\af|\leq d\}$.
When $\A=\N_d^n$, the $\A$-TKMP is specialized to
the standard truncated $K$-moment problem (TKMP),
and the $\A$-tms is called a tms of degree $d$.
Curto and Fialkow originally studied TKMPs
and have made foundational work in the field.
We refer to \cite{CF96,CF05,CF09} and the references therein.
Here we give a short review for TKMP.

Every tms $z \in \re^{\N_d^n}$ defines a {\it Riesz functional $\mathscr{L}_z$}
acting on $\re[x]_d$
(the space of real polynomials in $x:=(x_1,\ldots,x_n)$ of degrees at most $d$) as
\be  \label{def:Riesz}
\mathscr{L}_z \Big( \sum_{ \af \in \N_d^n} p_\af x^\af \Big)
:= \sum_{ \af \in \N_d^n} p_\af z_\af.
\ee
For convenience, denote $\langle p, z \rangle := \mathscr{L}_z(p)$.
We say that $\mathscr{L}_z$ is {\it $K$-positive} if
\[
\mathscr{L}_{z}(p) \ge 0 \quad \forall \,
p\in \mathbb{R}[x]_{d}: \, p|_K \ge 0.
\]
The $K$-positivity of $\mathscr{L}_z$ is necessary
for $z$ to admit a $K$-measure.
When $K$ is compact, it is also sufficient, which can be implied from the proof of
Tchakaloff's Theorem \cite{Tch}.
However, it is typically very difficult to check whether
a Riesz functional is $K$-positive or not.

A more favorable condition than $K$-positivity is {\it flatness}.
For convenience of description, suppose $K=\re^n$.
Denote $X\succeq 0$ (resp., $X\succ 0$) if the matrix $X$
is symmetric positive semidefinite (resp., definite).
For a tms $z \in \re^{\N_{2k}^n}$, define $M_k(z)$ to be
the symmetric matrix, which is linear in $z$, such that
\[
\mathscr{L}_z(p^2) = p^T  M_k(z)   p \qquad
\forall \, p\in \re[x]_k.
\]
(For convenience, we also use $p$ to denote the vector of coefficients of $p(x)$
in the graded lexicographical ordering.)
The matrix $M_k(z)$ is called a $k$-th order {\it moment matrix}.
If $z$ admits a $K$-measure $\mu$, then
\[
p^T  M_k(z)  p = \mathscr{L}_z(p^2) =
\int_K p^2 \mathtt{d}\mu \geq 0
\quad \forall \, p\in \re[x]_k.
\]
This implies that
\be \label{MOM>=0}
M_k(z) \succeq  0.
\ee
Hence, \reff{MOM>=0} is necessary for $z$ to admit a measure on $\re^n$,
but typically not sufficient.
However, if \reff{MOM>=0} is satisfied and $z$ is flat, i.e.,
\be \label{df:flat}
\rank \, M_{k-1}(z) = \rank \, M_k(z),
\ee
then $z$ admits a unique measure,
which is $r$-atomic with $r= \rank\,M_k(z)$.
When $K$ is a semialgebraic set as in \reff{def:K},
there is a similar version of this result (cf. Theorem~\ref{thm-CF:fec}).
This is an important result of Curto and Fialkow (cf. \cite{CF05}).
For convenience of notion, when $K=\re^n$, we simply say $z$ is {\it flat}
if it satisfies \reff{df:flat} and \reff{MOM>=0}.
(When $K$ is as in \reff{def:K}, we say $z$ is flat
if it satisfies \reff{loc-M>=0} and \reff{cond:FEC}.)

Flatness is very useful for solving truncated moment problems.
Let $y\in \re^{\N_d^n}, z\in \re^{\N_{e}^n}$ be two tms'.
We say $y$ is a {\it truncation} of $z$,
or equivalently, $z$ is an {\it extension} of $y$,
if $d \leq e$ and $y_\af = z_\af$ for all $\af \in \N_d^n$.
We denote by $z|_d$ the subvector of $z$,
whose entries are indexed by $\af \in \N_d^n$.
So, $z$ is an extension of $y$ if and only if $y=z|_d$.
If $z$ is flat and extends $y$,
we say $z$ is a {\it flat extension} of $y$.
Similarly, if $w$ is an extension of $z$ and $z$ is flat,
then we say $z$ is a {\it flat truncation} of $w$.
Clearly, if $y$ has a flat extension $w$,
then $y$ and $w$ commonly admit a finitely atomic measure
(cf. Theorem~\ref{thm-CF:fec}). Indeed,
Curto and Fialkow \cite{CF05} further proved that:
a tms $y \in \re^{\N_{d}^n}$ admits a $K$-measure if and only if
it has a flat extension $w \in\re^{\N_{2k}^n}$ (cf.~Theorem~\ref{CF:flat-extn}).
However, little is known on how to get flat extensions.

When $K$ is compact as in \reff{def:K},
Helton and Nie \cite{HN04} proposed a semidefinite approach
for solving TKMPs. Its basic idea is to find flat extensions
through semidefinite relaxations with moment and localizing matrices.
They proved that:
if $y \in \re^{\N_{d}^n}$ admits no $K$-measures,
then a certificate can be obtained for the nonexistence; if
there exists $\mu \in meas(y,K)$ such that
$\supp{\mu} \subseteq \mc{Z}(p):=\{x\in \re^n \mid p(x)=0\}$,
where $p$ is a polynomial having certain weighted SOS representations
and $\mc{Z}(p)$ is finite, then a flat extension of $y$ can be found.
However, for more general cases, no much was known on how
to get flat extensions.
%
%

\subsection{Contributions}

This paper studies the new and broader class of moment problems: $\A$-TKMPs.
The CP/SOEP-decomposition problems and standard
truncated $K$-moment problems are special cases of $\A$-TKMPs.

For an $\A$-tms $y$, a tms $w \in \re^{N_{2k}^n}$ is an extension of $y$
(or $y$ is a truncation of $w$) if
$w_\af = y_\af$ for all $\af \in \A$.
Denote by $w|_{\A}$ the subvector of $w$
whose indices belong to $\A$. Clearly, if $w|_{\A}=y$ and $w$ is flat,
then $w$ and $y$ commonly admit a finitely atomic measure.
When does $y$ admit a $K$-measure?
If it does, how can we get such a measure?
If it does not, can we get a certificate for the nonexistence?
They are the main questions in $\A$-TKMPs.

This paper proposes a numerical algorithm (i.e., Algorithm~\ref{sdpalg:A-tkmp})
for solving $\A$-TKMPs. It is based on solving a hierarchy of semidefinite relaxations
(we denote by $\{(\mathtt{SDR})_k\}_{k=1}^\infty$ here for convenience),
for a moment optimization problem whose objective
is a Riesz functional $\langle R, w\rangle:=\mathscr{L}_w(R)$.
The objective $R$ is generated in a certain randomized way.
Assume $K$ is given as in \reff{def:K} and is compact.
Denote $\re[x]_{\A} := \mbox{span}\{x^\af:\, \af \in \A\}$.
If $y$ admits no $K$-measures and $\re[x]_{\A}$ is $K$-full
(i.e., there exists $p \in \re[x]_{\A}$
that is positive on $K$), then $(\mathtt{SDR})_k$ will be infeasible for all $k$ big enough.
This gives a certificate for the nonexistence of a $K$-representing measure for $y$.
If $y$ admits a $K$-measure, then, for {\it almost all} generated $R$,
this algorithm has the following properties:

\bit

\item [i)] We can asymptotically get a flat extension of $y$
by solving the hierarchy $\{(\mathtt{SDR})_k\}_{k=1}^\infty$.
So, the convergence is guaranteed with probability one.

\item [ii)]  We can get a flat extension of $y$
by solving $(\mathtt{SDR})_k$ for some $k$,
under a general condition that is almost sufficient and necessary.
This implies that the finite convergence is very likely to happen.

\item [iii)] The obtained flat extensions admit a $r$-atomic $K$-measure
with $r\leq |\A|$.

\eit
\noindent
In all our numerical experiments, the finite convergence was always observed,
and we got $r$-atomic $K$-representing measures
with $r\leq |\A|$ when they exist.

CP/SOEP-decomposition problems are special cases of $\A$-TKMPs.
So, this algorithm can be applied to solve them.
If such decompositions do not exist, then the resulting $(\mathtt{SDR})_k$
will be infeasible for all big $k$,
which gives a certificate for the nonexistence;
if they exist, we can get such decompositions,
either asymptotically or in finitely many steps (very likely to happen).
This algorithm can also solve the standard TKMPs.
In the author's best knowledge, this is the first numerical algorithm
that has the aforementioned properties.

The paper is organized as follows.
Section~\ref{sec:back} reviews some backgrounds;
Section~\ref{sec:Atkmp-pro} presents some properties of $\A$-TKMPs;
Section~\ref{sec:sdalg} describes the algorithm;
Section~\ref{sec:cvganly} proves its convergence properties;
Section~\ref{sec:appl} gives applications in
CP/SOEP-decomposition problems and the standard TKMPs.

\section{Backgrounds}  \label{sec:back}
\setcounter{equation}{0}

\noindent
{\bf Notation} \,
The symbol $\N$ (resp., $\re$) denotes the set of nonnegative integers (resp., real numbers).
For $t\in \re$, $\lceil t\rceil$ (resp., $\lfloor t\rfloor$)
denotes the smallest integer not smaller
(resp., the largest integer not greater) than $t$.
For $x\in \re^n$, denote by
$
[x]_{\A} = (x^\af)_{ \af \in \A }
$
the vector of monomials, whose exponents are from $\A$,
ordered in the graded lexicographical ordering.
Denote $\N_d^n:=\{\af\in \N^n: |\af|\leq d\}$.
When $\A=\N_d^n$, the vector $[x]_{\A}$ is denoted by $[x]_d$.
Let $[k]:=\{1,\ldots,k\}$.
%
%
The symbol $\re[x] := \re[x_1,\ldots,x_n]$
denotes the ring of polynomials in $x:=(x_1,\ldots,x_n)$ with real coefficients.
%
%

For a set $S \subseteq \re^n$, $|S|$ denotes its cardinality,
and $int(S)$ denotes its interior.
The symbol $\mc{P}_d(S)$ denotes the set of
polynomials in $\re[x]_d$ that are nonnegative on $S$.
The superscript $^T$ denotes the transpose of a matrix.
For $u\in \re^N$, denote $\| u \|_2 := \sqrt{u^Tu}$
and $B(u,r) :=\{x\in \re^n:\, \|x-u\|_2 \leq r\}$.
For a polynomial $p\in \re[x]$, $\|p\|_2$ denotes the $2$-norm
of the coefficient vector of $p$.
Denote by $\mathbb{S}^k=\{x\in \re^{k+1}:\, \|x\|_2 =1 \}$
the $k$-dimensional unit sphere.
Denote by $\mc{S}^N$ the space of $N\times N$ real symmetric matrices.
For a matrix $A$, $\| A \|_2$ denotes its standard operator $2$-norm,
and $\|A\|_F$ denotes the standard Frobenius norm of $A$.
%

\subsection{Standard truncated $K$-moment problems}

Let $z \in \re^{ N_d^n }$ and $K\subseteq \re^n$.
Bayer and Teichmann \cite{BT} proved that:
$z$ admits a $K$-measure $\mu$ if and only if
it admits a $r$-atomic $K$-measure $\nu$ with $r\leq \binom{n+d}{d}$.
A nice exposition for this result can be found in Laurent \cite[Theorem~5.8]{Lau}.
When $K$ is compact, we can characterize the
existence of representing measures via Riesz functionals.
This can be implied by the proof of Tchakaloff's Theorem \cite{Tch}.

\begin{theorem}[Tchakaloff] \label{thm:Tchak}
Let $K$ be a compact set in $\re^n$. A tms $z \in \re^{ \N_d^n }$
admits a $K$-measure if and only if its
Riesz functional $\mathscr{L}_z$ is $K$-positive.
\end{theorem}

When $K$ is noncompact, the above might not be true.
A stronger condition is that
$\mathscr{L}_y$ is {\it strictly $K$-positive}, i.e.,
\[
\mathscr{L}_y(p) > 0 \quad \forall \,
p \in \re[x]_d: p|_K \geq 0, \, p|_K \not\equiv 0.
\]
When $K$ is a determining set (i.e., $p\equiv 0$ whenever $p|_K\equiv 0$),
if $\mathscr{L}_z$ is strictly $K$-positive, then $z$ admits a $K$-measure
(cf. \cite[Theorem~1.3]{FiNi1}).
Typically, checking $K$-positivity or strict $K$-positivity is quite difficult.

A more useful condition than $K$-positivity
is the positive semidefiniteness of localizing matrices.
For $z\in \re^{ \N_{2k}^n }$ and $h\in \re[x]_{2k}$,
define $L_{q}^{(k)}(z)$ to be the symmetric matrix,
which is linear in $z$, such that
\be \label{df:Loc-Mat}
\mathscr{L}_z(qp^2) \, = \, p^T
\left( L_{q}^{(k)}(z) \right)   p
\quad \forall \, p \in \re[x]_{ k - \lceil \deg(q)/2 \rceil }.
\ee
The matrix $L_{q}^{(k)}(z)$ is called the  $k$-th
{\it localizing matrix} of $q$ generated by $z$.
When $q=1$,  $L_{q}^{(k)}(z)$ coincides with the moment matrix $M_k(z)$.

Let $K$ be as in \reff{def:K} and $g_0 :=1$.
If $z \in \re^{ \N_{2k}^n }$ admits a $K$-measure $\mu$, then
\be  \label{loc-M>=0}
L_{h_i}^{(k)}(z) = 0 \, ( i =1,\ldots, m_1), \quad
L_{g_j}^{(k)}(z) \succeq 0  \, ( j =0,1,\ldots, m_2).
\ee
This is because for all such $i,j$
\[
p^T L_{h_i}^{(k)}(z)p  = \mathscr{L}_z(h_ip^2)
= \int_K h_ip^2 \mathtt{d}\mu = 0,
\]
\[
p^T L_{g_j}^{(k)}(z)p  = \mathscr{L}_z(g_jp^2)
= \int_K g_jp^2 \mathtt{d}\mu \geq 0,
\]
for all $p \in \re[x]$ with $\deg(h_ip^2),\deg(g_jp^2)\leq 2k$.
Therefore, \reff{loc-M>=0} is necessary for $z$
to admit a $K$-measure. Typically, it is not sufficient.  Let
\be \label{df:dg}
d_K  := \max_{i \in [m_1],  j \in [m_2]}
\{1, \lceil \deg(h_i)/2 \rceil, \lceil \deg(g_j)/2 \rceil\} .
\ee
If, in addition to \reff{loc-M>=0}, $z$ satisfies the rank condition
\be \label{cond:FEC}
\rank \, M_{k-d_K}(z) = \rank \, M_k(z),
\ee
then $z$ admits a unique $K$-measure, which is finitely atomic.
This is an important result of Curto and Fialkow.
For convenience of notion, we simply say $z$ is flat
if $z$ satisfies \reff{loc-M>=0} and \reff{cond:FEC}.

\begin{theorem}[\cite{CF05}] \label{thm-CF:fec}
Let $K$ be as in \reff{def:K}. If $z \in \re^{ \N_{2k}^n }$ is flat (i.e.,
\reff{loc-M>=0} and \reff{cond:FEC} hold),
then $z$ admits a unique $K$-measure,
which is $\rank\,M_k(z)$-atomic.
\end{theorem}

Flatness is very useful for solving truncated moment problems,
as shown by Curto and Fialkow \cite{CF96,CF982}.
A nice exposition for this can also be found in Laurent \cite{Lau05}.
For a flat tms, its finitely atomic representing measure can be found
by solving some eigenvalue problems,
as shown by Henrion and Lasserre~\cite{HenLas05}.

Clearly, if a tms $y$ is not flat but has a flat extension $z$,
then $y$ admits a $K$-measure.
Thus, the existence of a $K$-representing measure for $y$
can be determined by investigating whether
$y$ has a flat extension or not. Indeed, Curto and Fialkow \cite{CF05}
proved the following result (the set $K$ does not need to be compact).

\begin{theorem}[\cite{CF05}]\label{CF:flat-extn}
Let $K$ be as in \reff{def:K}.
Then a tms $y \in \re^{ \N_d^n }$ admits a $K$-measure if and only if
it has a flat extension $z \in \re^{ \N_{2k}^n }$.
\end{theorem}

There are other necessary conditions for admitting $K$-representing measures,
like the recursively generated relation.
We refer to \cite{CF91,CF96,CF982}.

\subsection{Ideals, varieties and positive polynomials}

A subset $I \subseteq \re[x]$ is called an {\it ideal} if
$I + I \subseteq I$ and $p\cdot q \in I$ for all $p\in I$ and $q \in \re[x]$.
For a tuple $p:=(p_1,\ldots,p_m)$ of polynomials in $\re[x]$,
denote by $I(p)$ the ideal generated by $p_1,\ldots, p_m$,
which is the set
$p_1 \re[x] + \cdots + p_m \re[x]$.
A variety is a subset of $\cpx^n$ that consists of
common zeros of a set of polynomials.
A real variety is the intersection of a variety with $\re^n$.
For the tuple $p=(p_1,\ldots,p_m)$, denote
\[
V_{\cpx}(p) := \{v\in \cpx^n \mid \, p(v) = 0 \}, \quad
V_{\re}(p)  := \{v\in \re^n \mid \, p(v) = 0 \}.
\]

A polynomial $f\in \re[x]$ is a {\it sum of squares} (SOS)
if there exist $f_1,\ldots,f_k \in \re[x]$ such that
$f=f_1^2+\cdots+f_k^2$.
The set of all SOS polynomials in $n$ variables and of degree $d$
is denoted by $\Sig_{n,d}$.
It is a convex cone in $\re^{\N_d^n}$
and has nonempty interior for any even $d>0$.
We refer to Reznick~\cite{Rez00} for SOS polynomials.

Let $K,h,g$ be as in \reff{def:K}. Denote $g_0 :=1$ and
\be \label{t-ideal:h}
I_{2k}(h) :=
\left\{
\left.  \sum_{i=1}^{m_1} h_i \phi_i \right|
 \mbox{ each }  \deg(h_i\phi_i) \leq 2k
\right\},
\ee
\be \label{qmod-g}
Q_k(g) :=
\left\{
\left. \sum_{i=0}^{m_2} g_i \sig_i  \right|
\baray{c}
\mbox{ each } \deg(\sig_ig_i) \leq 2k \\
\mbox{ and } \sig_i \mbox{ is SOS}
\earay
\right\}.
\ee
Clearly, $I(h) = \cup_{ k\in \N} \, I_{2k}(h)$.
The set $I_{2k}(h)$ is called a truncation of $I(h)$ with degree $2k$.
The union $Q(g):= \cup_{ k\in \N} \, Q_k(g)$
is called the {\it quadratic module} generated by $g$.
Each $Q_k(g)$ is a truncation of $Q(g)$ with degree $2k$.

Clearly, if $f\in I(h) +  Q(g)$, then $f$ is nonnegative on $K$.
For compact $K$ as in \reff{def:K},
if a polynomial $f$ is positive on $K$, then
$f$ is not necessarily in $f\in I(h) +  Q(g)$,
but we have $f\in I(h) + Pre(g)$
($Pre(g)$ is the preordering generated by $g$).
This was proved by Schm\"{u}dgen \cite{Smg}.
However, if, in addition, the archimedean condition holds for $(h,g)$
(i.e., $N-\|x\|_2^2 \in I(h) + Q(g)$ for some $N$),
then we have $f\in I(h) + Q(g)$.
This was shown by Putinar \cite{Put}.

\begin{theorem}[Putinar, \cite{Put}]\label{thm:PutThm}
Let $K$ be as in \reff{def:K}. Suppose the archimedean condition holds for $(h,g)$.
If $f\in \re[x]$ is positive on $K$, then $f\in I(h) + Q(g)$.
\end{theorem}

In the literature,
Theorem~\ref{thm:PutThm} is called {\it Putinar's Positivstellentsatz}
(cf. Lasserre's book\cite[Theorem~2.14]{LasBok}, Laurent's survey\cite[p.186]{Lau}),
or {\it Representation Theorem}
(cf. Marshall's book~\cite[Theorem~5.4.4]{MarBk}).
Various versions of this result were also found by
Dubois, Jacobi, Kadison, Krivine, Stone, etc
(cf.~\cite[p.79]{MarBk}).
When $K$ is compact, the archimedean condition can be
satisfied by adding a redundant condition,
like $N-\|x\|_2^2 \geq 0$ for $N$ big enough.

\subsection{Semidefinite optimization with moment variables}

Semidefinite programming (SDP) is very useful in solving moment problems.
We refer to \cite{Todd} for SDP,
and refer to \cite{LasBok,Lau,HN04} for semidefinite programs arising
from moment problems.

Let $\mc{S}_+^N$ be the cone of positive semidefinite matrices in $\mc{S}^N$.
Let $\mc{K} := \mc{S}^{N_1}_+ \times \cdots \times \mc{S}^{N_\ell}_+$
be a cone in the space $\mc{S}^{N_1} \times \cdots \times \mc{S}^{N_\ell}.$
A tuple $X=(X_1,\ldots,X_\ell)$ belongs to $\mc{K}$
if and only if each $X_i\in \mc{S}^{N_i}_+$.
A general semidefinite program is
\be \label{Ksdp:prm}
\min \quad  Trace(CX) \quad
s.t. \quad \mc{F}(X) = f, \,  X  \in \mc{K}
\ee
with $C$ a constant tuple in
$\mc{S}^{N_1} \times \cdots \times \mc{S}^{N_\ell}$,
$\mc{F}$ a linear operator
from $\mc{S}^{N_1} \times \cdots \times \mc{S}^{N_\ell}$
to $\re^m$, and $f$ a vector in $\re^m$.
Let $\mc{F}^*$ be the adjoint operator of $\mc{F}$.
The dual optimization problem of \reff{Ksdp:prm} is
\be \label{Ksdp:dual}
\max \quad  f^Ty \quad
s.t. \quad \mc{F}^*(y) + Z =C, \, Z \in \mc{K}.
\ee

Let $K,h,g$ be as in \reff{def:K}. Denote
\be \label{df:Phi(g)}
\Phi_k(g) := \left\{
\left. w \in \re^{ \N_{2k}^n } \right|
L_{g_j}^{(k)}(w) \succeq 0, \,\,\,
j=0,1,\ldots, m_2 \right\},
\ee
\be \label{df:Ek(h)}
E_k(h) := \left\{
\left. w \in \re^{ \N_{2k}^n } \right|
L_{h_i}^{(k)}(w) = 0, \ \,\, \,
i=1,\ldots,m_1 \right\}.
\ee
It can be shown that $\Phi_k(g)$ is the dual cone of $Q_k(g)$
and  $E_k(h)$ is the dual cone of $I_{2k}(h)$ (cf. \cite{LasBok,Lau}).
Indeed, $E_k(h)$ is a subspace of $\re^{ \N_{2k}^n }$.

A typical semidefinite program in this paper is
\be  \label{mom:sdrlx}
\underset{ w }{\min} \quad c^Tw  \quad
s.t. \quad w|_{\A} = y, \,\, w \in \Phi_k(g) \cap E_k(h)
\ee
with given $c\in \re^{ \N_d^n }$ ($\deg(\A)<d\leq 2k$) and $y \in \re^{\A}$.
The dual optimization problem of \reff{mom:sdrlx} is
\be \label{sosrlx:mom}
\underset{ p \in \re[x]_{\A} }{\max} \quad \langle p, y \rangle \quad
s.t. \quad c-p \in  I_{2k}(h) + Q_k(g).
\ee
Any objective value of a feasible solution of  \reff{mom:sdrlx}
(resp., \reff{sosrlx:mom}) is an upper bound (resp., lower bound)
for the optimal value of the other one
(this is called {\it weak duality}).
If one of them has an interior point
(for \reff{sosrlx:mom} it means that there exists $p\in \re[x]_{\A}$
such that $c-p$ lies in the interior of
$\big(Q_k(g) + I_{2k}(h)\big) \cap \re[x]_d$,
and for \reff{mom:sdrlx} it means that there is a feasible
$w$ with $L_{g_j}^{(k)}(w) \succ 0$ for all $j$),
then the other one has an optimizer and they have the same optimal value
(this is called {\it strong duality}).
We refer to \cite[Section~2.4]{BTN} for duality theory.

\section{Properties of $\A$-TKMPs} \label{sec:Atkmp-pro}
\setcounter{equation}{0}

This section presents some properties of
$\A$-truncated $K$-moment problems.
Recall that $\A$ is a finite set in $\N^n$.

\subsection{$\A$-Riesz functionals}

An $\A$-tms $y$ defines an $\A$-Riesz functional
$\mathscr{L}_y$ that acts on $\re[x]_{\A}$ as
\[
\mathscr{L}_y \Big( \sum_{\af\in\A} p_\af x^\af \Big ) =
\sum_{\af\in\A} p_\af  y_\af.
\]
The $\A$-Riesz functional $\mathscr{L}_y$
is said to be {\it $K$-positive} if
\[
\mathscr{L}_y(p) \geq 0 \quad \forall \,
p \in  \re[x]_{\A}:\, p|_K \geq 0,
\]
and {\it strictly} $K$-positive if
\[
\mathscr{L}_y(p) > 0 \quad \forall \,
p \in  \re[x]_{\A}:\, p|_K \geq 0, \, p|_K \not\equiv 0.
\]
Clearly, if $y$ admits a $K$-measure $\mu$, then $\mathscr{L}_y$
must be $K$-positive because
\[
\mathscr{L}_y(p) = \int_K p \mathtt{d} \mu  \geq 0
\]
whenever $p \in  \re[x]_{\A}$ and $p|_K \geq 0$.
So, the $K$-positivity of $\mathscr{L}_y$ is a necessary condition
for $y$ to admit a $K$-measure.
Indeed, it is also a sufficient if $K$ is compact and $\re[x]_{\A}$ is {\it $K$-full}
(i.e., there exists $p\in \re[x]_{\A}$ such that
$p>0$ on $K$). This is a result of Fialkow and Nie \cite{FiNi2012}.

\begin{theorem}(\cite[Theorem~2.2]{FiNi2012}) \label{thm2.2:NF03}
Let $\A\subseteq \N^n$ be a finite set.
Suppose $K\subseteq \mathbb{R}^n$ is compact and $\re[x]_{\A}$ is $K$-full.
Let $y$ be an $\A$-tms such that $\mathscr{L}_y$ is $K$-positive.
Then, there exist $N \le \dim \re[x]_{\A}$, $u_{1},\ldots,u_{N} \in K$,
and $c_{1},\ldots,c_{N} > 0$, such that
\[
y =  c_1 [u_1]_{\A} + \cdots + c_N [u_N]_{\A}.
\]
\end{theorem}

The above theorem immediately implies the following.

\begin{cor}  \label{cor:ctf-nomeas}
Suppose $K\subseteq \mathbb{R}^n$ is compact and $\re[x]_{\A}$ is $K$-full.
If an $\A$-tms $y$ admits no $K$-measures,
then there exists $p\in \re[x]_{\A}$ such that
\[
p|_K \geq 0, \quad  \mathscr{L}_y(p)<0.
\]
\end{cor}

There is a similar version of Theorem~\ref{CF:flat-extn} for $\A$-TKMPs.
Recall that a tms $w\in \re^{N_{2k}^n}$ is called flat
if it satisfies \reff{loc-M>=0} and \reff{cond:FEC}.

\begin{prop} \label{pro:atms:extn}
(i) Let $K\subseteq \mathbb{R}^n$  be a set.
Then, an $\A$-tms $y$ admits a $K$-measure if and only if
it admits a $r$-atomic $K$-measure with $r\leq |\A|$. \\
(ii) Let $K$ be as in \reff{def:K}.
Then, an $\A$-tms $y$ admits a $K$-measure
if and only if $y$ is extendable to a flat tms $w\in \re^{N_{2k}^n}$ for some $k$.
\end{prop}
\begin{proof}
(i) The ``if" direction is obvious.  The ``only if" direction
can be proved by a formal repetition of the proof of Theorem~5.8 in Laurent \cite{Lau}.
The inequality $r\leq |\A|$ is implied by Carath\'{e}odory's Theorem.
Here we omit it for cleanness of presentation.

(ii)
``$\Rightarrow$"\, Suppose $y$ admits a measure on $K$.
By item (i), it admits a $r$-atomic $K$-measure, say,
$
y = c_1 [u_1]_{\A} + \cdots + c_r [u_r]_{\A}
$
with all $u_i\in K$, $c_i>0$ and $r\leq |\A|$.
For $k$ big enough, the tms
$
w = c_1 [u_1]_{2k} + \cdots + c_r [u_r]_{2k} \in \re^{\N_{2k}^n}
$
is flat. Clearly, $w$ is an extension of $y$.

``$\Leftarrow$" Suppose $w\in \re^{\N_{2k}^n}$ is flat and $w|_{\A}=y$.
By Theorem~\ref{thm-CF:fec}, $w$ admits a $K$-measure $\mu$ and
$w_{\af} = \int_K x^\af \mathtt{d}\mu $ for all $\af \in \N_{2k}^n$. Since $w|_{\A}=y$,
we have $y_{\af} = \int_K x^\af \mathtt{d}\mu$ for all $\af \in \A$, i.e.,
$y$ admits the $K$-measure $\mu$.
\end{proof}

\subsection{Extremal extensions}

For $d>\deg(\A)$, define the $d$-th extension of $y$ as
\be \label{Ed(y,K)}
\mc{E}_d(y,K) =
\left\{ z \in \mathscr{R}_{\N_{d}^n}(K): \, z|_{\A} = y  \right\}.
\ee
Clearly, $\mc{E}_d(y,K)$ is convex.
If $meas(y,K)\ne\emptyset$, then $\mc{E}_d(y,K) \ne \emptyset$.

\begin{lem}  \label{lm:E-compact}
Let $y\in \re^{\A}$. If $K$ is compact and $\re[x]_{\A}$ is $K$-full,
then $\mc{E}_d(y,K)$ is a compact convex set.
\end{lem}
\begin{proof}
The compactness of $K$ implies that $\mathscr{R}_{\N_{d}^n}(K)$ is closed,
which can be implied by Theorem~\ref{thm:Tchak}.
So, $\mc{E}_d(y,K)$ is also closed.
If $meas(y,K) =\emptyset$, then $\mc{E}_d(y,K) = \emptyset$ and we are done.
Thus, assume $meas(y,K)\ne\emptyset$, then $\mc{E}_d(y,K) \ne \emptyset$.
We need to prove that $\mc{E}_d(y,K)$ is bounded.
Since $\re[x]_{\A}$ is $K$-full and $K$ is compact,
there exist $p \in \re[x]_{\A}$ and $\eps$ such that
$p|_{K} \geq \eps >0.$ There exists $M>0$ such that for all $x\in K$
\[
-M \leq x^\af \leq M \quad \forall  \af \in \N_d^n.
\]
For all $z\in \mc{E}_d(y,K)$, the $K$-positivity of $\mathscr{L}_z$ implies that
\[
-M z_0  \leq z_\af \leq M z_0 \quad \forall  \af \in \N_d^n.
\]
Similarly, $p|_{K} \geq \eps$ implies
$\mathscr{L}_y(p) \geq \eps z_0$. So,
\[
|z_\af| \leq M \mathscr{L}_y(p)/\eps  \quad \forall  \af \in \N_d^n.
\]
Hence, $\mc{E}_d(y,K)$ is bounded, which completes the proof.
\end{proof}

\begin{lem} \label{lm:extrm=>fin}
Let $K\subseteq \re^n$, $d>\deg(\A)$ and $y \in \mathscr{R}_{\A}(K)$.
Suppose $z$ is an extreme point of $\mc{E}_d(y,K)$.
If $\mu \in meas(z,K)$, then $\mu$ is $r$-atomic with $r \leq | \A |$.
\end{lem}
\begin{proof}
Choose an arbitrary $\mu \in meas(z,K)$.
If $\mu$ is $r$-atomic and $r \leq | \A |$,
then we are done. We derive a contradiction if
either $\mu$ is $r$-atomic with $r > | \A |$,
or $\mu$ is not finitely atomic.

First, consider the case that $\mu$ is $r$-atomic with $r > |\A|$, say,
\[
\mu = c_1 \dt(u_1)  + \cdots + c_r \dt(u_r),
\]
with distinct points $u_1, \ldots, u_r \in K$
and $c_1,\ldots, c_r>0$.
Here, $\dt(v)$ denotes the Dirac measure supported on the point $v$.
For each $\af \in \N_d^n$, denote
\[
w^{(\af)} := ( u_1^{\af}, \ldots,  u_r^{\af})^T \in \re^r.
\]
We show that there must exist $\bt \in \N_{d}^n \backslash \A$ such that the system
\be  \label{at=0btX=0}
\left(w^{(\af)}\right)^T t = 0 \,( \forall \, \af \in \A),  \quad
\left(w^{(\bt)}\right)^T t \ne 0
\ee
has a solution $t=(t_1,\ldots, t_r)$. Suppose otherwise,
then for all $\bt \in \N_{d}^n \backslash \A$,
\[
\left(w^{(\af)}\right)^T t = 0 \,( \forall \, \af \in \A)
\quad \Longrightarrow \quad
\left(w^{(\bt)}\right)^T t  = 0.
\]
This implies that, for every $\bt \in \N_{d}^n \backslash \A$,
the vector $w^{(\bt)}$ is a linear combination of
the vectors $w^{(\af)}(\af\in\A)$, i.e.,
there exist real numbers $p_{\bt,\af}$ such that
\[
w^{(\bt)} = \sum_{\af\in \A}  p_{\bt,\af}  w^{(\af)}.
\]
For each $\bt \in \N_{d}^n \backslash \A$, let
\[
p_\bt:= x^\bt - \sum_{\af\in \A}  p_{\bt,\af}  x^{\af}.
\]
Note that $p_\bt(u_i)=0$ for each $i=1,\ldots,r$.
Let $J$ be the ideal generated by polynomials
$
p_{\bt} ( \bt \in \N_{d}^n \backslash \A ).
$
We show that the ideal $J$ is zero-dimensional, i.e., the quotient space
$\re[x]/J$ is finitely dimensional.
For each $\bt \in \N_{d}^n \backslash \A$,
$x^{\bt} - p_{\bt} \in \re[x]_{\A}$, so
$x^{\bt}$ is equivalent to a polynomial in $\re[x]_{\A}$ modulo $J$.
Since $\deg(A) < d$, $\re[x]_{\A} \subseteq \N_{d-1}^n$.
For each $\eta \in \N_d^n$,
we have $x^\eta = x^\bt x_i$ for some $i\in [n]$ and
$\bt \in \N_{d}^n \backslash \A$. So,
\[
x^\eta \equiv   x_i\sum_{\af\in \A}  p_{\bt,\af}  x^{\af}
= \sum_{\af\in \A}  p_{\bt,\af}  x^{\af+e_i} \quad \mod \quad J.
\]
Because $\af+e_i \in \N_d^n$ for all $\af\in \A$,
$x^{\af+e_i}$ is equivalent to a polynomial
in $\re[x]_{\A}$ modulo $J$, and so is $x^{\eta}$.
Repeating this process and by induction,
we can show that for all $\zeta \in \N^n$ with $\zeta \not\in \A$,
$x^{\zeta}$ is equivalent to a polynomial
in $\re[x]_{\A}$ modulo $J$. This means that the quotient space
$\re[x]/J$ is finitely dimensional, and its dimension
\[
D:=\dim \re[x]/J \leq \dim \re[x]_{\A} = |\A|.
\]
By Proposition~2.1 of Sturmfels~\cite{Stu02},
the number of common zeros of polynomials $p_{\bt}$
($\bt \in \N_{d}^n \backslash \A$) is equal to $D$, counting multiplicities.
Hence, the cardinality of $V_{\cpx}(J)$ is at most $|\A|$.
However, the distinct points $u_1,\ldots,u_r$ all belong to $V_{\cpx}(J)$
and $r>D$, which is a contradiction.
So, there exists $t=(t_1,\ldots, t_r)$ satisfying \reff{at=0btX=0}
for some $\bt \in \N_{d}^n \backslash \A$. We have
\[
(c_1,\ldots, c_r) \pm \eps (t_1,\ldots, t_r) > 0
\]
for $\eps >0$ sufficiently small. Let
{\small
\[
\mu_1 := \sum_{i=1}^r (c_i + \eps t_i) \dt(u_i),\quad
\mu_2 := \sum_{i=1}^r (c_i - \eps t_i) \dt(u_i).
\]
} \noindent
They are all nonnegative Borel measures supported in $K$. Let
\[
z_1 = \int_K [x]_d \mathtt{d} \mu_1, \quad
z_2 = \int_K [x]_d \mathtt{d} \mu_2.
\]
Then, both $z_1$ and $z_2$ belong to $\mc{E}_d(y,K)$ and
$z=\half (z_1+z_2)$. The inequality in \reff{at=0btX=0} implies that
$z_1\ne z,  \, z_2 \ne z.$
This contradicts that $z$ is extreme in $\mc{E}_d(y,K)$.

Second, consider the case that $\mu$ is not finitely atomic.
Then $|\supp{\mu}|=+\infty$.
Choose $|\A|+1$ distinct points, say, $v_1, \ldots, v_{|\A|+1}$, from $\supp{\mu}$.
The support of $\mu$ is the smallest closed set $S$
such that $\mu(\re^n\backslash S)=0$ (cf.~\cite[Section~4]{Lau}).
So, there exists $\eps >0$ such that
$\mu(B(v_i, \eps))>0$ for all $i$ and the balls
$B(v_1,\eps),\ldots, B(v_{|\A|+1}, \eps)$ are disjoint from each other. Let
$
T_i := B(v_i,\eps) \cap \supp{\mu}
$
for $i=1,\ldots, |\A|$, and
{\small
\[
T_{|\A|+1} := \supp{\mu} \, \backslash \bigcup_{i=1}^{|\A|} B(v_i, \eps).
\]
} \noindent
This results in the decomposition
\[
\supp{\mu} = T_1 \cup \cdots \cup T_{|\A|+1}.
\]
Note that
$\mu(T_1) >0, \ldots, \mu(T_{|\A|+1}) >0$ and
$T_i \cap T_j = \emptyset$ whenever $i \ne j$.
For each $j=1,\ldots,|\A|+1$, let $\mu_j = \mu|_{T_j}$, the restriction of $\mu$ on $T_j$.
Then
\[
z = \int_K [x]_d \mathtt{d} \mu_1 + \cdots \int_K [x]_d \mathtt{d} \mu_{|\A|+1}.
\]
Each tms $\int_K [x]_d \mathtt{d} \mu_i$ admits a finitely atomic
measure supported in $T_i$ (cf. \cite[Corollary 5.9]{Lau}).
Hence, there exists a measure $\theta \in meas(z,K)$ that is
$r$-atomic with $r>|\A|$.
Therefore, a contradiction can be obtained as in the first case.

Combining the above two cases, we know the conclusion is true.
\end{proof}

\subsection{Linear optimization over $\mc{E}_d(y,K)$}

Let $d>\deg(\A)$ and $R \in \re[x]_d$.
Consider the linear moment optimization problem
\be  \label{mom:min<R,z>}
\min_{ z  } \quad  \langle R, z \rangle  \quad
s.t. \quad z|_{\A} = y, z \in \mathscr{R}_d(K).
\ee
The feasible set of \reff{mom:min<R,z>} is $\mc{E}_d(y,K)$.
Clearly, if $z^*$ is a unique minimizer of \reff{mom:min<R,z>},
then $z^*$ is an extreme point of $\mc{E}_d(y,K)$.
When $K$ is a compact set,
the cone $\mathcal{P}_d(K)=\{p\in \re[x]_d:\, p|_K \geq 0\}$
is the dual cone of $\mathscr{R}_d(K)$.
This was shown by Tchakaloff \cite{Tch} (also see Laurent \cite[Section~5.2]{Lau}).
Thus, for compact $K$, the dual problem of \reff{mom:min<R,z>} is
\be \label{psd:max<p,y>}
\max_{ p\in \re[x]_{\A} } \quad \langle p, y \rangle \quad
s.t. \quad R - p \in \mathcal{P}_d(K).
\ee

\begin{pro} \label{pr:optimizer}
Let $K\subseteq \re^n$ be a compact set
and $y$ be an $\A$-tms in $\mathscr{R}_{\A}(K)$.
\bit

\item [(i)] If either $R|_K>0$ or $\re[x]_{\A}$ is $K$-full,
then \reff{mom:min<R,z>} and \reff{psd:max<p,y>}
have the same optimal value and \reff{mom:min<R,z>} has a minimizer.

\item [(ii)] If $\mathscr{L}_y$ is strictly $K$-positive,
then \reff{psd:max<p,y>} has a maximizer.

\eit
\end{pro}
\begin{proof}
(i) If $R|_K>0$, then the origin is an interior point of \reff{psd:max<p,y>}.
If $\re[x]_{\A}$ is $K$-full, then there exists
$q \in \re[x]_{\A} \subseteq \re[x]_d$ that is
strictly positive on $K$, so \reff{psd:max<p,y>} also has an interior point.
The problem \reff{mom:min<R,z>} is feasible, because $y \in \mathscr{R}_{\A}(K)$.
Thus, the strong duality holds and the conclusion is true
(cf. \cite[Section~2.4]{BTN}).

(ii) From $y \in \mathscr{R}_{\A}(K)$, we know $y$ has a flat extension
$w \in \re^{ \N_{2k}^n }$ for some $k$ (cf.~Proposition~\ref{pro:atms:extn}).
Then the truncation $w|_d$ is a feasible point for \reff{mom:min<R,z>}.
By weak duality, the optimal value of \reff{psd:max<p,y>} is finite, say, $\eta$.
Clearly, the feasible set of \reff{psd:max<p,y>} is closed.
Let $\{p_k\} \subseteq \re[x]_{\A}$ be a sequence such that
each $R-p_k \in \mc{P}_d(K)$ and
\[
\langle p_k, y \rangle \to  \eta, \quad \mbox{ as } \, k \to \infty.
\]
Let
$
S_1 = \{ f \in \re[x]_{\A}: \, f|_K \equiv 0 \},
$
and $S_2$ be the orthogonal complement of $S_1$ in $\re[x]_{\A}$. So,
$\re[x]_{\A} = S_1 + S_2$,  $S_1 \perp S_2$.
Because $\langle f, y \rangle =0$ and $f|_K \equiv 0$ for all $f\in S_1$,
the above $p_k$ can be chosen such that $p_k \in S_2$
and $R-p_k \in \mc{P}_d(K)$.

If the sequence $\{p_k\}$ is bounded, then any of its accumulation points
is a maximizer of \reff{psd:max<p,y>} and we are done.
Suppose otherwise $\{p_k\}$ is unbounded, say, $\|p_k\|_2 \to \infty$.
Let $\hat{p}_k = p_k / \|p_k\|_2$ be the normalization.
The sequence $\{\hat{p}_k\}$ is bounded, and we can generally assume
$\hat{p}_k \to p^* \in S_2$. Clearly, $\|p^*\|_2=1$ and
\[
\langle \hat{p}_k, y \rangle =
\langle p_k, y \rangle / \|p_k\|_2 \to  0, \quad \mbox{ as } \, k \to \infty.
\]
This implies that $\langle p^*, y \rangle = 0$.
From $(R-p_k)/\|p_k\|_2 \in \mc{P}_d(K)$, we get $-p^*|_K \geq 0$.
Since $p^* \in S_2$ and $\|p^*\|_2=1$, we know
$-p^*|_K \not\equiv 0$. The strict $K$-positivity of $\mathscr{L}_y$
implies $\langle -p^*, y \rangle > 0$, which is a contradiction.
Thus, the sequence $\{p_k\}$ must be bounded,
and the proof is complete.
\end{proof}

\section{A semidefinite algorithm for $\A$-TKMPs}   \label{sec:sdalg}
\setcounter{equation}{0}

In this section, we present a numerical algorithm
for solving $\A$-truncated $K$-moment problems.
To determine whether an $\A$-tms $y$ admits a $K$-measure or not,
by Proposition~\ref{pro:atms:extn}, it is equivalent to
investigating whether $y$ has a flat extension or not.
If it does not exist, then $y$ does not admit a $K$-measure.
If it exists, then we can get a finitely atomic representing measure for $y$.

The extension set $\mc{E}_d(y,K)$, defined in \reff{Ed(y,K)},
is very useful in getting flat extensions. By Lemma~\ref{lm:extrm=>fin},
all $K$-measures admitted by extreme points of $\mc{E}_d(y,K)$
are $r$-atomic with $r\leq |\A|$.
Clearly, if \reff{mom:min<R,z>} has a unique minimizer $z^*$,
then $z^*$ is an extreme point of $\mc{E}_d(y,K)$.
A very useful fact is that when we minimize a generic linear function
over a compact convex set, the minimizer is unique (cf.~\cite[Theorem~2.2.4]{Sneid}).
%
%
If $K$ is compact and $\re[x]_{\A}$ is $K$-full, the set $\mc{E}_d(y,K)$
is compact convex by Proposition~\ref{lm:E-compact},
and \reff{mom:min<R,z>} has a minimizer for all $R$.
If $\re[x]_{\A}$ is not $K$-full,
we typically need to choose $R$ that is positive definite on $K$,
to guarantee \reff{mom:min<R,z>} has a minimizer.
Therefore, to get an extreme point of $\mc{E}_d(y,K)$,
it is enough to solve \reff{mom:min<R,z>}
for a generic positive definite $R$,
no matter $\re[x]_{\A}$ is $K$-full or not.

The cone $\mathscr{R}_d(K)$ is typically quite difficult to describe,
and generally we can not solve \reff{mom:min<R,z>} directly.
Recently, there is much work on solving moment optimization problems
like \reff{mom:min<R,z>} by semidefinite relaxations.
We refer to \cite{FiNi2012,HN04,Las08,LasBok}.
The basic idea is to approximate the cone $\mathscr{R}_d(K)$ by semidefinite programs.
So, we apply semidefinite relaxations to solve \reff{mom:min<R,z>}.
This produces a semidefinite algorithm for solving $\A$-TKMPs.

Suppose $K$ is compact and $K\subseteq B(0,\rho)$.
Let $b:=\rho^2 - \|x\|_2^2$. Recall the polynomial tuples $h$ and $g$
for describing $K$ as in \reff{def:K}. For convenience, denote
\[
g_B := (g_1, \ldots, g_{m_2}, b).
\]
The set $K$ can be equivalently described as
$h(x)=0, g_B(x) \geq 0$.
Recall the definitions of $Q_k(g)$ in \reff{qmod-g}
and its dual cone $\Phi_k(g)$ in \reff{df:Phi(g)}. Note that
\[
Q_k(g_B) = Q_k(g) + Q_k(b), \quad
\Phi_k(g_B) = \Phi_k(g) \cap \Phi_k(b).
\]

Let $k \geq d/2$ be an integer.
The $k$-th order semidefinite relaxation of \reff{mom:min<R,z>} is
\be \label{qmod:min<R,w>}
\min_{ w }
\quad  \langle R, w \rangle \quad \quad s.t.  \quad \quad
w|_{\A} = y, \, w \in \Phi_k(g_B) \cap E_k(h).
\ee
The dual optimization problem of \reff{qmod:min<R,w>} is
\be  \label{putsos:max<p,y>}
\max_{ p \in \re[x]_{\A} }
\quad  \langle p, y \rangle \quad
s.t.  \quad R - p \in I_{2k}(h) + Q_k(g_B).
\ee
For every $w$ feasible for \reff{qmod:min<R,w>}
and every $p$ feasible for \reff{putsos:max<p,y>}, we have
\[
\langle R, w \rangle \geq  \langle p, y \rangle,
\]
by weak duality. Thus, the optimal value of \reff{qmod:min<R,w>} is always
greater than or equal to that of \reff{putsos:max<p,y>}.

\begin{pro} \label{pro:alg-atkmp}
Let $y \in \re^{\A}$, $d> \deg(\A)$, and
$K\subseteq B(0,\rho)$ be as in\reff{def:K}.

\bit

\item [(i)] If $w^*$ is a minimizer of \reff{qmod:min<R,w>}
and has a flat truncation $w^*|_{2t}$ with $2t\geq \deg(\A)$,
then $y$ admits a finitely atomic $K$-measure $\mu$.

\item [(ii)] If \reff{qmod:min<R,w>} is infeasible for some $k$,
then $y$ admits no $K$-measures.

\eit

\end{pro}
\begin{proof}
(i) Let $z=w^*|_{2t}$. Then $z$ is a flat tms.
So, $z$ admits a finitely atomic $K$-measure $\mu$, by Theorem~\ref{thm-CF:fec}.
Since $2t\geq \deg(\A)$, $z|_{\A}=y$ and $z$ is an extension of $y$.
Thus, $y$ also admits the measure $\mu$.

(ii) Suppose otherwise $y$ admits a $K$-measure.
By Proposition~\ref{pro:atms:extn},
$y$ can be extended to a flat tms $w_1 \in \re^{\N_{2k_1}^n}$ satisfying \reff{loc-M>=0}.
By Theorem~\ref{thm-CF:fec}, the tms $w_1$ admits a $r$-atomic $K$-measure, say,
$
w_1 = c_1 [u_1]_{2k_1} + \cdots + c_r [u_r]_{2k_1}
$
with all $c_i>0$ and $u_i \in K$. Let
$w = c_1 [u_1]_{2k} + \cdots + c_r [u_r]_{2k}$,
then $w$ is feasible for \reff{qmod:min<R,w>},
a contradiction. So, $y$ admits no $K$-measures.
\end{proof}

Proposition~\ref{pro:alg-atkmp} can be applied to determine whether
an $\A$-tms $y$ admits a $K$-measure or not.
We can start with an order $k \geq d/2$.
If \reff{qmod:min<R,w>} is infeasible, then we know $y$
admits no $K$-measures, by Proposition~\ref{pro:alg-atkmp} (ii).
If \reff{qmod:min<R,w>} is feasible, we solve it for a minimizer $w^*$ if it exists.
If $w^*$ has a flat truncation $w^*|_{2t}$ with $2t\geq \deg(\A)$,
then $y$ and $w^*|_{2t}$ commonly admit
a finitely atomic $K$-measure.
If such a flat truncation does not exist,
we increase $k$ by one, and repeat the above.
This produces the following algorithm.

\begin{alg} \label{sdpalg:A-tkmp}
A semidefinite algorithm for $\A$-TKMPs. \\
\noindent
{\bf Input:}\, An $\A$-tms $y$, an even degree $d>\deg(\A)$,
a semialgebraic set $K$ as in \reff{def:K} and $\rho>0$ with $K\subseteq B(0,\rho)$.

\noindent
{\bf Output:}\, A finitely atomic $K$-representing measure for $y$,
or an answer that it does not exist.

\noindent
{\bf Procedure:} \quad
\bdes

\item [Step 0] Choose a generic $R \in \Sig_{n,d}$ and let $k:=d/2$.

\item [Step 1] Solve \reff{qmod:min<R,w>}.
If \reff{qmod:min<R,w>} is infeasible,
output the answer that $y$ admits no $K$-measures, and stop.
If \reff{qmod:min<R,w>} is feasible,
get a minimizer $w^{*,k}$. Let $t := \min\{d_K, \deg(\A)\}$.

\item [Step 2] Let $z:=w^{*,k}|_{2t}$.
Check whether the rank condition~\reff{cond:FEC} is satisfied or not.
If yes, go to Step~3; otherwise, go to Step~4.

\item [Step 3] Compute the finitely atomic measure $\mu$ admitted by $z$:
\[
\mu = c_1 \dt(u_1) + \cdots + c_r \dt(u_r),
\]
where $r=\rank M_t(z)$, each $u_i \in K$ and $c_i >0$.
Output $\mu$, and stop.

\item [Step 4] If $t<k$, set $t:=t+1$ and go to Step~2;
otherwise, set $k:=k+1$ and go to Step~1.

\edes

\end{alg}

For the input, we typically choose $d=2\lceil (\deg(\A)+1)/2 \rceil$,
which is the minimum as required.
In Step~0, the genericity means that $R$ is chosen in $\Sig_{n,d}\backslash \Theta$,
for a set $\Theta \subseteq \re^{\N_d^n}$ having zero Lebesgue measure.
In computations, we can choose $R$ as $[x]_{d/2}^TG^TG[x]_{d/2}$,
with $G$ a random square matrix obeying Gaussian distribution.
In Step~2, the rank condition \reff{cond:FEC} is usually checked
by using numerical ranks, due to computer round-off errors.
In our numerical experiments, we evaluate the rank of a matrix
as the number of its singular values that are greater than or equal to $10^{-6}$.
In Step~3, the method in \cite{HenLas05} can be used
to get a $r$-atomic $K$-measure for $z$.
Algorithm~\ref{sdpalg:A-tkmp} can be easily implemented
by using software {\tt GloptiPoly~3} \cite{GloPol3}
and {\tt SeDuMi} \cite{sedumi}.
Example~\ref{em:prob-square} shows how to do this.

\begin{remark}  \label{rmk:min-spt}
If $y$ admits a $K$-measure, Algorithm~\ref{sdpalg:A-tkmp} can
produce a $r$-atomic $K$-representing measure with $r\leq |\A|$,
for {\it almost all} $R \in \Sig_{n,d}$,
either asymptotically or in finitely many steps (cf. Section~\ref{sec:cvganly}).
The obtained $r$ may not be minimum.
It is typically quite difficult to find a representing measure
whose support is minimum. This is an important future work,
and we do not focus on it here.
However, this question can be treated in some way.
For instance, Algorithm~\ref{sdpalg:A-tkmp} can be applied repeatedly
for a certain number of times, say, $N$.
In each time, a different $R$ is generated,
and we typically get different $r$-atomic measures.
Among these $N$ times, we choose the measure
whose support is minimum. Our numerical experiments show that
this often produces a $r$-atomic measure with
$r$ equal or close to the minimum.
Of course, this is heuristic, and there is no theoretical guaranty.
\end{remark}

\begin{exm}  \label{em:prob-square}
Consider the $\A$-tms $y$ with $(\af, y_\af)$ given as:
\bcen
\btab{|c|c|c|c|c|c|c|c|} \hline
$\af$ &  $\bpm 2\\0\epm$ & $\bpm 0\\2\epm$ & $\bpm 2\\ 1\epm$ &
$\bpm 1\\ 2 \epm$ & $\bpm 2\\ 2\epm$ & $\bpm 4 \\ 2 \epm$ & $\bpm 2\\ 4 \epm$  \\  \hline
$y_{\af}$ & 1/3 & 1/3 & 0 & 0 & 1/9 & 1/15 & 1/15 \\ \hline
\etab.
\ecen
This $y$ admits a measure supported on the square $[-1,1]^2$,
since each $y_\af$ is the average of $x^\af$ on $[-1,1]^2$.
Clearly, $h=\emptyset$, $g=(1-x_1^2,1-x_2^2)$ and $[-1,1]^2 \subseteq B(0,\sqrt{2})$.
In \reff{qmod:min<R,w>}, the tuple $g_B$ can be replaced by $g$,
because $\Psi_k(g_B) = \Psi_k(g)$.
We apply Algorithm~\ref{sdpalg:A-tkmp} to this $\A$-TKMP.
The order $k=4$ is typically enough to get a flat extension.
This can be done by the syntax in {\tt GloptiPoly~3} \cite{GloPol3}
and {\tt SeDuMi} \cite{sedumi} as follows:
\bcen
\begin{verbatim}
   mpol x  2;
   Amon = [x(1)^2  x(2)^2  x(1)^2*x(2)  x(1)*x(2)^2  ...
   x(1)^2*x(2)^2  x(1)^4*x(2)^2  x(1)^2*x(2)^4];
   y=[ 1/3  1/3  0  0  1/9  1/15  1/15];
   conmom = [mom(Amon)==y];
   K = [1-x(1)^2>=0, 1-x(2)^2>=0];
   bracx = mmon(x,0,4);  G = randn(length(bracx));
   R = bracx'*(G'*G)*bracx;  k = 4;
   P = msdp(min(mom(R)),K,conmom,k);
   [A,b,c,S] = msedumi(P);
   [xsol,ysol,info] = sedumi(A,b,c,S);
   dvar = c-A'*ysol;
   Mw = mat( dvar(S.f+1:S.f+S.s(1)^2) );
\end{verbatim}
\ecen
In the above, {\tt Mw} is the moment matrix $M_k(w)$.
By Remark~\ref{rmk:min-spt}, we run Algorithm~\ref{sdpalg:A-tkmp} for a couple of times.
In each time, the computed {\tt Mw}
satisfies the rank condition \reff{cond:FEC}, and we got a $r$-atomic measure
with $r \leq |\A|=7$, by the method in \cite{HenLas05}.
The smallest $r$ we got is $3$, which occurs in the representing measure
$\Sig_{i=1}^3 c_i \dt(u_i)$ with $u_i$ and $c_i$ given as
\footnote{Throughout the paper, only four decimal digits are shown
for supports and weights.}
\bcen
\btab{|c|c|c|c|} \hline
$u_i$ &  (-0.8524, 0.8910)  &  (0.2109, 0.8873)  &  (0.6697, -0.3743)   \\  \hline
$c_i$ &  0.1231 & 0.2061 &  0.5233 \\ \hline
\etab.
\ecen
\end{exm}

\begin{exm}  \label{exm:pb:sphere}
Consider the $\A$-tms $y$ with $(\af, y_\af)$ given as:
\bcen
\btab{|c|c|c|c|c|c|c|} \hline
$\af$ & $\bpm 4 \\ 0 \\ 0 \epm$ & $\bpm  2 \\ 0 \\2 \epm$ &   $\bpm 0 \\ 2 \\ 2 \epm$ &
$\bpm 4 \\ 0 \\ 2 \epm$ &  $\bpm 2 \\ 2 \\ 2\epm$ & $\bpm 0 \\ 0 \\ 6\epm$ \\  \hline
$y_{\af}$ & 1/5 & 1/15 & 1/15 & 3/105 & 1/105  & 1/7  \\ \hline
\etab.
\ecen
The above $y$ admits a measure on $\mathbb{S}^2$,
because each $y_\af$ is the average of $x^\af$ on $\mathbb{S}^2$.
The sphere $\mathbb{S}^2$ is defined by $h=(\|x\|_2^2-1)$ and $g=\emptyset$.
Clearly, $\mathbb{S}^2 \subseteq B(0,1)$.
In \reff{qmod:min<R,w>}, the tuple $g_B=(1-\|x\|_2^2)$ can be replaced by $g=\emptyset$,
because $\Psi_k(g_B) \cap E_k(h) = \Psi_k(g) \cap E_k(h)$.
Like in Example~\ref{em:prob-square}, we run Algorithm~\ref{sdpalg:A-tkmp}
for a couple of times. In each time, we got a $r$-atomic measure with $r\leq |\A|=6$.
The smallest $r$ we got is $2$, occurring in the representing measure
$\Sig_{i=1}^2 c_i \dt(u_i)$ with $u_i,c_i$ given as
\bcen
\btab{|c|c|c|} \hline
$u_i$ & (0.3434    0.4542    0.8221) & (0.8999    0.2550   -0.3539) \\ \hline
$c_i$ &      0.4610     &    0.2952 \\ \hline
\etab.
\ecen
\end{exm}

\begin{exm} (random instances)   \label{exm:rand:ball}
We apply Algorithm~\ref{sdpalg:A-tkmp}
to solve some randomly generated $\A$-TKMPs.
Let $K=B(0,1)$ be the unit ball in $\re^n$.
For each pair $(n,m)$ from $\{(2,10),(3,8),(4,6),(5,4)\}$,
we randomly generate a subset
$\A \subseteq \N_m^n$ with cardinality $10, 20, 30$ respectively.
For each $\A$, generate $N:=\binom{n+m}{m}$ points randomly from the ball $B(0,1)$,
say, $u_1, \ldots, u_N$, and let
$y = c_1 [u_1]_{\A}+\cdots+ c_N [u_N]_{\A}$, with $c_i>0$ random.
This is because if a tms in $\re^{ \N_m^n }$ admits a measure,
then it must admit an $N$-atomic measure (cf.~\cite{BT}).
\bcen
\btab{|c|c|c|c|} \hline
$(n,m)$  &  $|\A|=10$ & $|\A|=20$ & $|\A|=30$ \\ \hline
(2, 10) & 4,5,6,7,8  &  8,9,10,11,12,13    &  9,10,11,12,13,14,15  \\   \hline
(3, 8) & 4,5,6,7,8,9   &  9,10,11,12,13,14,15    &  11,12,13,14,15,16,17,18  \\   \hline
(4, 6) & 5,6,7,8,9   &   8,9,10,11,12,13,14    &  10,11,12,13,14,15,16,17     \\   \hline
(5, 4) & 3,4,5,6,7,8 &   7,8,9,10,11,12         &  10,11,12,13,14,15,16  \\  \hline
\etab
\ecen
For each triple $(n,m,|\A|)$, we generate $100$ random instances.
For every generated instance, Algorithm~\ref{sdpalg:A-tkmp}
returned a $r$-atomic representing measure.
The values of obtained $r$ are listed in the above table.
They are all smaller than $|\A|$.
This is justified by Proposition~\ref{pr:alg:!min}(ii).
\end{exm}

\section{Convergence Analysis} \label{sec:cvganly}
\setcounter{equation}{0}

In this section, we analyze the convergence of Algorithm~\ref{sdpalg:A-tkmp}.
Two kinds of convergence will be investigated:
{\it asymptotic} convergence, and {\it finite} convergence.
For asymptotic convergence, we mean that there exists $t$ such that
the truncated sequence $\{w^{*,k}|_{2t}\}$
($w^{*,k}$ is a minimizer of \reff{qmod:min<R,w>} with order $k$)
is bounded and all its accumulation points are flat extensions of $y$,
if $y$ admits a $K$-measure.
For finite convergence, we mean that there exists $k$ such that,
either \reff{qmod:min<R,w>} is infeasible, or there exists $t$ such that
the truncation $w^{*,k}|_{2t}$ is flat.
We begin with some properties of semidefinite relaxations
\reff{qmod:min<R,w>}-\reff{putsos:max<p,y>}.

\subsection{Properties of semidefinite relaxations}

\begin{pro}  \label{pro:sdprlx}
Let $d> \deg(\A)$ be even, and $K\subseteq B(0,\rho)$ be as in\reff{def:K}.
Suppose $y \in \re^{\A}$ admits a $K$-measure.

\bit

\item [(i)] If $R$ lies in the interior of $\Sig_{n,d}$,
then, for all $k \geq d/2$, \reff{qmod:min<R,w>} is feasible and has a minimizer,
and \reff{qmod:min<R,w>}-\reff{putsos:max<p,y>} have the same optimal value.

\item [(ii)] Let $R_{min}$ be the optimal value of \reff{mom:min<R,z>}.
Then, the optimal values of \reff{qmod:min<R,w>} and \reff{putsos:max<p,y>}
are less than or equal to $R_{min}$,
for all $k \geq d/2$.

\item [(iii)] Suppose $R$ lies in the interior of $\Sig_{n,d}$.
Then, there exists a constant $C=C(R)$ such that, for all
$w$ that is a minimizer of \reff{qmod:min<R,w>} with order $k$,
\be \label{w2t<=rhoR}
\|w|_{2t}\|_2 \leq  (1+\rho^2+\cdots+\rho^{2t}) C, \quad
t=0,1,\ldots,k.
\ee

\eit
\end{pro}
\begin{proof}
(i) Let $\mu$ be a finitely atomic $K$-representing measure for $y$,
which must exist by Proposition~\ref{pro:atms:extn}.
Then the tms $\int_K [x]_{2k} \mathtt{d}\mu$ is feasible for \reff{qmod:min<R,w>}.
So, \reff{qmod:min<R,w>} is feasible for all $k \geq d/2$.
Since $R \in int(\Sig_{n,d})$,
for all $p \in \re[x]_{\A}$ with tiny coefficients,
$R-p \in \Sig_{n,d} \subseteq I_{2k}(h) + Q_k(g_B)$.
This means that the zero polynomial $0$
is an interior point of \reff{putsos:max<p,y>}.
Hence, the strong duality holds, i.e., \reff{qmod:min<R,w>} has a minimizer,
and \reff{qmod:min<R,w>}-\reff{putsos:max<p,y>} has the same optimum value,
for all $k\geq d/2$ (cf. \cite[Section~2.4]{BTN}).

(ii) Since $y$ admits a $K$-measure, the feasible set of \reff{mom:min<R,z>}
is nonempty. Let $z$ be an arbitrary feasible point of \reff{mom:min<R,z>}.
Then $z$ admits a finitely atomic $K$-measure $\mu$, i.e., $z=\int_K [x]_d \mathtt{d}\mu$.
Let $w = \int_K [x]_{2k} \mathtt{d}\mu$. Clearly, $w|_{\A}=y$
and is feasible for \reff{qmod:min<R,w>},
and $\langle R, z \rangle = \langle R, w \rangle$.
That is, \reff{qmod:min<R,w>} is a relaxation of \reff{mom:min<R,z>}.
So, the optimal value of \reff{qmod:min<R,w>} is at most $R_{min}$.
This is also true for the optimal value of \reff{putsos:max<p,y>},
since it is not bigger than that of \reff{qmod:min<R,w>}.

(iii) Since $R \in int(\Sig_{n,d})$,
$R-\eps \in \Sig_{n,d}$ for some $\eps>0$. So, we have
\[
0 \leq \langle R - \eps, w \rangle =
\langle R, w \rangle - \eps \langle 1, w \rangle.
\]
(Cf.~\cite[Lemma~2.5]{Nie-FT}.) Since $w$ is a minimizer of \reff{qmod:min<R,w>},
item (ii) implies that
\[
w_0 = \langle 1, w \rangle \leq
\langle R, w \rangle /\eps \leq R_{min}/\eps.
\]
The membership $w \in \Phi_k(g_B)$ implies $L_b^{(k)}(w) \succeq 0$
($b=\rho^2-\|x\|_2^2$). So,
\[
\rho^2 \mathscr{L}_w(\|x\|_2^{2t-2}) - \mathscr{L}_w(\|x\|_2^{2t}) \geq 0,
\quad t = 0, 1, \ldots, k.
\]
(Cf.~\cite[Lemma~2.5]{Nie-FT}.)
A repeated application of the above gives
\[
\mathscr{L}_w(\|x\|_2^{2t}) \leq \rho^{2t} w_0,
\quad t = 0, 1, \ldots, k.
\]
Since $M_k(w) \succeq 0$, for each $t=0,1,\ldots,k$,
\[
\|w|_{2t}\|_2 \leq \|M_t(w)\|_F \leq Trace(M_t(w)) =
\sum_{i=0}^t \sum_{|\af| = i}\, w_{2\af},
\]
\[
\sum_{|\af| = i}\, w_{2\af} = \mathscr{L}_w(\sum_{|\af| = i} \,x^{2\af} )
 \leq \mathscr{L}_w(\|x\|_2^{2i}) \leq  \rho^{2i} w_0.
\]
Let $C=R_{min}/\eps$, then the inequality \reff{w2t<=rhoR} holds.
\end{proof}

In Step~0, the genericity means that $R$ is chosen in $\Sig_{n,d}\backslash \Theta$,
for a set $\Theta \subseteq \re^{\N_d^n}$ having zero Lebesgue measure.

\begin{prop}  \label{pr:alg:!min}
Let $K\subseteq \re^n$ be compact and $d > \deg(\A)$ be even.
Suppose $y\in \re^{\A}$ admits a $K$-measure.
If $R$ is generic in $\Sig_{n,d}$ (i.e., $R \in \Sig_{n,d}\backslash \Theta$,
for a subset $\Theta \subseteq \Sig_{n,d}$ having zero Lebesgue measure),
then we have:
\bit
\item [(i)] The problem \reff{mom:min<R,z>} has a unique minimizer.

\item [(ii)] If, for some $k$, a minimizer $w^{*,k}$ of \reff{qmod:min<R,w>}
has a flat truncation $w^{*,k}|_{2t}$ ($2t\geq d$),
then the measure admitted by $w^{*,k}|_{2t}$ is $r$-atomic with $r\leq |\A|$.

\eit
\end{prop}
\begin{proof}
The boundary of $\Sig_{n,d}$ has zero
Lebesgue measure in the space $\re^{\N_d^n}$.
It is enough to prove the items (i) and (ii)
if $R$ is generic in the interior of $\Sig_{n,d}$. Let
\[
S_\ell:=\{ R \in \Sig_{n,d}: \,  R - 1/\ell \in \Sig_{n,d}, \|R\|_2 \leq \ell \}
\quad (\ell = 1, 2, \ldots ).
\]
Clearly, $int(\Sig_{n,d}) = \bigcup_{\ell \geq 1} int(S_\ell)$.
It is sufficient to prove that for each
$\ell =1, 2, 3, \ldots$, if $R$ is generic in
$S_\ell$, then the items (i) and (ii) are true.

(i) For every $R \in S_\ell$, it holds that for all $x\in \re^n$
\[
1/\ell  \leq  R(x) \leq \ell \|[x]_d\|_2.
\]
Choose a finitely atomic measure $\nu^* \in meas(y,K)$. Let
$
M_1 := \int_K  \ell \|[x]_d\|_2 \, \mathtt{d}\nu^*,
$
a constant independent of $R$.
Clearly, $\int_K [x]_d \mathtt{d}\nu^*$ is feasible in \reff{mom:min<R,z>}, and
\[
R_{min} \leq \int_K  R  \mathtt{d}\nu^*  \leq  M_1.
\]
For all $R \in S_\ell$, \reff{mom:min<R,z>} has a minimizer $z^*$
(cf.~Proposition~\ref{pr:optimizer}(i)), and $z^*$ satisfies
\[
(z^*)_0/\ell  \leq  \langle R, z^* \rangle  = R_{min}  \leq  M_1.
\]
Note that $z^*$ is feasible for \reff{qmod:min<R,w>} with order $k=d/2$.
As in the proof of Proposition~\ref{pro:sdprlx}(iii), we can get
\[
\|z^*\|_2 \leq M_2:=(1+\rho^2+\cdots + \rho^d) \ell M_1.
\]
Thus, for all $R \in S_\ell$, \reff{mom:min<R,z>} is equivalent to
\be  \label{min<R,z>:z<=M}
\min_{z} \quad  \langle R, z \rangle  \quad
s.t. \quad z|_{\A} = y, \, \|z\|_2 \leq M_2, \, z \in \mathscr{R}_d(K).
\ee
The feasible set of \reff{min<R,z>:z<=M}, denoted by $F$,
is a nonempty compact convex set.
A linear functional $\langle R, z \rangle$
has a unique minimizer on $F$ if and only if
$R$ is a regular normal vector of $F$,
or equivalently, $\langle R, z \rangle$
has more than one minimizer on $F$ if and only if
$R$ is a singular normal vector of $F$ (cf. Schneider~\cite[Section~2.2]{Sneid}).
Let $\Theta$ be the set of all singular normal vectors of $F$.
Then $\Theta$ has zero Lebesgue measure in the space $\re^{\N_d^n}$
(cf. Schneider~\cite[Theorem~2.2.4]{Sneid}).
If $R \in S_\ell\backslash\Theta$,
then \reff{min<R,z>:z<=M}, as well as \reff{mom:min<R,z>},
has a unique minimizer. So, if $R$ is generic in $S_\ell$,
then \reff{mom:min<R,z>} has a unique minimizer.

(ii) Since $w^{*,k}|_{2t}$ is flat, it admits a finitely atomic $K$-measure,
and so does $w^{*,k}|_d$ when $2t \geq d$.
Thus, $w^{*,k}|_d$ is feasible in \reff{mom:min<R,z>},
and $\langle R, w^{*,k}|_d \rangle \geq R_{min}$.
By Proposition~\ref{pro:sdprlx}(ii),
we know $\langle R, w^{*,k}|_d  \rangle  = \langle R, w^{*,k} \rangle  \leq R_{min}$.
Hence, $\langle R, w^{*,k}|_d  \rangle  = R_{min}$
and $w^*|_d$ is a minimizer of \reff{mom:min<R,z>}.
By item (i), if $R$ is generic in $\Sig_{n,d}$,
then \reff{mom:min<R,z>} has a unique minimizer.
Therefore, for a generic $R \in \Sig_{n,d}$,
$w^*|_d$ is the unique minimizer of \reff{mom:min<R,z>}
and is an extreme point of $\mc{E}_d(y,K)$.
By Lemma~\ref{lm:extrm=>fin},
every measure admitted by $w^{*,k}|_{d}$ must be $r$-atomic with $r\leq |\A|$.
Clearly, every measure admitted by $w^{*,k}|_{2t}$
is also admitted by $w^{*,k}|_{d}$,
and thus must be $r$-atomic with $r\leq |\A|$.
\end{proof}

\subsection{Asymptotic convergence}

Our main result in this subsection is:

\begin{theorem}  \label{cvgthm:asymp}
Let $d> \deg(\A)$ be even and $K\subseteq B(0,\rho)$ be as in\reff{def:K}.
Suppose $y \in \re^{\A}$ admits a $K$-measure.
In Algorithm~\ref{sdpalg:A-tkmp}, if $R$ is generic in $\Sig_{n,d}$,
then \reff{qmod:min<R,w>} has an optimizer $w^{*,k}$ for all $k \geq d/2$
and we have:

\bit

\item [(i)] For all $t$ big enough, the sequence $\{w^{*,k}|_{2t}\}$
is bounded and all its accumulation points are flat.
Moreover, each of its accumulation points
admits a $r$-atomic measure with $r\leq |\A|$.

\item [(ii)] In item (i), if, in addition, $d$ is also big enough,
then the sequence $\{w^{*,k}|_{2t}\}$ converges to a flat tms.

\eit
\end{theorem}

\begin{proof}
Since $R$ is generic in $\Sig_{n,d}$, we can assume $R \in int(\Sig_{n,d})$.
By Proposition~\ref{pro:sdprlx}(i),
\reff{qmod:min<R,w>} has an optimizer $w^{*,k}$ for every $k \geq d/2$.

(i) By Proposition~\ref{pro:sdprlx}(iii),
there is a constant $C=C(R)$ such that
\[
\|w^{*,k}|_{2t}\|_2 \leq  (1+\rho^2+\cdots+\rho^{2t}) C
\quad \mbox{ for all } \quad t =0,1,\ldots, k.
\]
So, the sequence $\{w^{*,k}|_{2t}\}$ is bounded, for any fixed $t$.
Let $\omega$ be an accumulation point of $\{w^{*,k}|_{2t}\}$.
We can generally further assume $w^{*,k}|_{2t} \to \omega$ as $k\to\infty$.

Without loss of generality, we can assume $\rho <1$.
\Big(If otherwise $\rho \geq 1$, we can do as follows.
Apply the scaling transformation
$x = (\rho+1) \tilde{x}$. Let $\tilde{K} =
\{ \tilde{x} \in \re^n: (\rho+1) \tilde{x} \in K\}$.
Then $\tilde{K} \subseteq B(0, \tilde{\rho})$ with
$\tilde{\rho} := \rho/(1+\rho)<1$. Define
the polynomial tuples $\tilde{h}, \tilde{g}_B$ in $\tilde{x}$ such that
$\tilde{h}(\tilde{x}) = h((\rho+1)\tilde{x})$,
$\tilde{g}_B(\tilde{x}) = g_B((\rho+1)\tilde{x})$.
For $w \in \re^{\N_{2k}^n}$ (resp., $y \in \re^{\A}$),
let $\tilde{w} \in \re^{\N_{2k}^n}$ (resp., $\tilde{y} \in \re^{\A}$)
be its scaling such that $w_\af = (\rho+1)^{|\af|}\tilde{w}_\af$
$\forall \af \in \N_{2k}^n$
(resp., $y_\af  = (\rho+1)^{|\af|}\tilde{y}_\af$ $\forall \af \in \A$).
Note that $y$ admits a $K$-measure if and only if
$\tilde{y}$ admits a $\tilde{K}$-measure.
Let $\tilde{R}\in \re[\tilde{x}]$ be such that
$\tilde{R}(\tilde{x}) = R((\rho+1)\tilde{x})$.
Then, $\langle R, w \rangle =
\langle \tilde{R}, \tilde{w}\rangle$, and
$w$ is feasible for \reff{qmod:min<R,w>} if and only if
its scaling $\tilde{w}$ is feasible for the scaled problem
\[
\min_{ \tilde{w} }
\quad  \langle \tilde{R}, \tilde{w} \rangle \quad s.t.  \quad
\tilde{w}|_{\A} = \tilde{y}, \, \tilde{w} \in
\Phi_k(\tilde{g}_B) \cap E_k(\tilde{h}).
\]
Let $\tilde{w}^{*,k}$ be the scaling of $w^{*,k}$, as in the above.
So, $w^{*,k}$ is an optimizer of \reff{qmod:min<R,w>}
if and only if $\tilde{w}^{*,k}$ is an optimizer of
the above scaled optimization problem.
The two sequences $\{w^{*,k}|_{2t}\},\{\tilde{w}^{*,k}|_{2t}\}
\subseteq \re^{\N_{2t}^n}$ are scaled from each other,
so they have same convergence properties. Therefore, the proof
for the case $\rho \geq 1$ can be equivalently reduced
to the case $\tilde{\rho} < 1$,
by applying the scaling procedure as above.\Big)

First, we prove that the truncation $\omega|_d$ is a minimizer of \reff{mom:min<R,z>}.
Let $\re^{\N_{\infty}^n}$ be the space of all vectors
indexed by $\af \in \N^n$. For vectors $u,v \in \re^{\N_{\infty}^n}$,
define their inner product as
\[
\langle u, v \rangle := \sum_{\af \in \N^n } u_\af v_\af.
\]
Let $\mathscr{H}(\re^{\N_{\infty}^n})$ be the Hilbert space
of vectors in $\re^{\N_{\infty}^n}$ whose norms are finite,
under the above inner product.
Each $w^{*,k}$ can be thought of as a vector in
$\mathscr{H}(\re^{\N_{\infty}^n})$ by adding zero entries to the tailing.
In the above, we have shown that for all $k$
\[
\| w^{*,k} \|_2 = \| w^{*,k}|_{2k} \|_2  \leq
(1 + \rho^2 + \cdots + \rho^{2k} ) C \leq C/(1-\rho).
\]
The sequence $\{ w^{*,k} \}$ is bounded in the Hilbert space $\mathscr{H}(\re^{\N_{\infty}^n})$.
So, it has a subsequence $\{w^{*,k_j} \}$ that is convergent in the weak-$\ast$ topology,
i.e., there exists $w^* \in \mathscr{H}(\re^{\N_{\infty}^n})$ such that
\[
\langle c, w^{*,k_j} \rangle \, \to \,  \langle c,  w^* \rangle
 \quad \mbox{ as } j \to \infty, \quad
\forall \, c \in \mathscr{H}(\re^{\N_{\infty}^n}).
\]
This fact can also be implied by Alaoglu's Theorem (cf.~\cite[Theorem~V.3.1]{Conway}
or \cite[Theorem~C.18]{LasBok}).
If we choose $c$ as $\langle c, w \rangle= w_\af$ for each $\af$, then
\be \label{2k2t->w*2t}
(w^{*,k_j})_{\af} \, \to  \, (w^*)_{\af}.
\ee
Since $w^{*,k}|_{2t} \to \omega$, the above implies $w^*|_{2t} = \omega$.
In particular, $w^*|_\A = y$. So, if $y \ne 0$, then $w^*$ cannot be a zero vector.
Note that
$
w^{*,k_j} \in \Phi_{k_j}(g_B) \cap E_{k_j}(h)
$
for all $k_j$. For each $r=1,2,\ldots$, if $k_j \geq 2r$, then
\[
L_{h_i}^{(r)}(w^{(k_j)}) = 0\,(1\leq i \leq m_1), \quad
L_{g_i}^{(r)}(w^{(k_j)})\succeq 0 \,(0\leq i \leq m_2).
\]
Hence, \reff{2k2t->w*2t} implies that for all $r=1,2,\ldots$
\[
L_{h_i}^{(r)}(w^*) = 0\,(1\leq i \leq m_1), \quad
L_{g_i}^{(r)}(w^*)\succeq 0 \,(0\leq i \leq m_2).
\]
This means that $w^*$ is a full multisequence
whose localizing matrices of all orders are positive semidefinite.
By Lemma~3.2 of Putinar \cite{Put}, $w^*$ admits a $K$-measure.
So, the truncation $\omega|_d = w^*|_d$ is feasible for \reff{mom:min<R,z>}, and
\[
R_{min} \leq \langle R, w^*|_d\rangle = \langle R, w^*\rangle
= \langle R, \omega|_d \rangle.
\]
By Proposition~\ref{pro:sdprlx}(ii), we know
$
\langle R, w^{*,k} \rangle \leq R_{min}
$
for all $k$. Thus,
\[
\langle R, w^* \rangle = \lim_{j\to\infty}
\langle R, w^{*,k_j} \rangle \leq R_{min}.
\]
Hence, $\langle R, \omega|_d \rangle = R_{min}$,
and $\omega|_d$ is a minimizer of \reff{mom:min<R,z>}.

Second, we prove that if $t \geq |\A|d_K$ then
the truncation $\omega|_{2t}$ is flat.
By Proposition~\ref{pr:alg:!min}, if $R$ is generic in $\Sig_{n,d}$,
then \reff{mom:min<R,z>} has a unique minimizer,
which must be $\omega|_d$, by the above.
So, $\omega|_d$ is an extreme point of $\mc{E}_d(y,K)$\footnote
{Suppose otherwise $\omega|_d$ is not an extreme point of $\mc{E}_d(y,K)$,
say, $\omega|_d = \lmd \omega^{(1)} + (1-\lmd) \omega^{(2)}$
with $0 < \lmd < 1$, $\omega|_d \ne \omega^{(1)}, \omega^{(2)} \in \mc{E}_d(y,K)$.
Clearly, $\langle R, \omega|_d \rangle =
\lmd \langle R, \omega^{(1)} \rangle + (1-\lmd) \langle R, \omega^{(2)} \rangle$.
Note that $\langle R, \omega^{(1)} \rangle, \langle R, \omega^{(2)} \rangle
\geq  \langle R, \omega|_d  \rangle$,
because $\omega|_d$ is a minimizer of \reff{mom:min<R,z>}.
So, $\langle R, \omega|_d \rangle =
 \langle R, \omega^{(1)} \rangle = \langle R, \omega^{(2)} \rangle$.
That is, $\omega^{(1)}, \omega^{(2)}$
are both minimizers of \reff{mom:min<R,z>},
which are different from $\omega|_d$.
However, this is a contradiction because \reff{mom:min<R,z>}
has a unique minimizer that is $\omega|_d$.
Hence, $\omega|_d$ must be an extreme point of $\mc{E}_d(y,K)$.},
which is the feasible set of \reff{mom:min<R,z>}.
Let $\mu^*$ be a $K$-representing measure for $\omega$,
which must exist because $\omega = w^*|_{2t}$.
Then $\mu^* \in meas(\omega|_d,K)$.
By Lemma~\ref{lm:extrm=>fin}, $\mu^*$ must be
finitely atomic and $|\supp{\mu^*}|\leq |\A|$.
Note that
\[
\rank\, M_0(\omega) \leq \rank\, M_{d_K}(\omega) \leq
\cdots \leq \rank\, M_{d_K|\A|}(\omega) \leq \cdots,
\]
\[
\rank\, M_i(\omega) \leq |\supp{\mu^*}|\leq |\A|
\quad (i=1,2,\ldots).
\]
There must exist $\ell \leq |\A|$ such that
\[
\rank\, M_{ (\ell-1)d_K }(\omega)  = \rank\, M_{ \ell d_K }(\omega).
\]
So, the truncation $\omega|_{2\ell d_K}$ is flat.
Clearly, $\omega$ is an extension of $\omega|_{2\ell d_K}$,
and every measure admitted by $\omega$
is also a representing measure for $\omega|_{2\ell d_K}$. By Theorem~\ref{thm-CF:fec},
$\omega|_{2\ell d_K}$ admits a unique $K$-representing measure,
which is $r$-atomic with $r=\rank\, M_{ \ell d_K }(\omega)$.
Because $\mu^* \in meas(\omega|_{2\ell d_K},K)$, $\mu^*$ is the such unique measure,
and it is $r$-atomic. From the above, we can get
\[
\baray{c}
r = |\supp{\mu^*}| = \rank\, M_{ (\ell-1) d_K }(\omega)
= \rank\, M_{ \ell d_K }(\omega) \leq \\
\rank\, M_{ \ell d_K +1 }(\omega) \leq \cdots \leq
\rank\, M_t(\omega) \leq |\supp{\mu^*}|=r.
\earay
\]
So, we must have $\rank\, M_{t-d_K}(\omega)=\rank\, M_t(\omega)$,
i.e., $\omega$ is flat.

Third, in the above, we have indeed shown that
if $\mu^*$ is a $K$-representing measure for any accumulation point $\omega$,
then $\mu^*$ is $r$-atomic with $r\leq |\A|$.

(ii) It is enough to show that if $t \geq d/2 \geq |\A| d_K$
then $\{w^{*,k}|_{2t}\}$ has a unique accumulation point.
We continue the proof of the item (i). By Proposition~\ref{pr:alg:!min},
if $R$ is generic in $\Sig_{n,d}$,
then \reff{mom:min<R,z>} has a unique minimizer, say, $z^*$.
Let $\omega$ be an arbitrary accumulation point of $\{w^{*,k}|_{2t}\}$.
In  the proof of (i),
we showed that $\omega|_d$ is a minimizer of \reff{mom:min<R,z>}
So, $z^*=\omega|_d$. We also showed that $\omega|_{2t}$
is flat for all $t \geq |\mc{A}| d_K$, in the proof of (i).
Thus, if $d \geq 2|\A| d_K$, then $z^* = \omega|_d$ is flat.
By Theorem~\ref{thm-CF:fec}, $z^*$ admits a unique $K$-representing measure,
say, $meas(z^*, K) = \{\nu^*\}$.
Since $z^*$ is a truncation of $\omega$, $meas(\omega,K) \subseteq meas(z^*, K)$.
As shown in the proof of (i), we know $\omega$ is flat, so
$meas(\omega,K) \ne \emptyset$.
%
%
Hence, we must have $meas(\omega, K) = \{\nu^*\}$
and $\omega = \int_K [x]_{2t} \mathtt{d}\nu^*$.
This shows that $\{w^{*,k}|_{2t}\}$ has a unique accumulation point,
which is $\int_K [x]_{2t} \mathtt{d}\nu^*$.
\end{proof}

Algorithm~\ref{sdpalg:A-tkmp} is guaranteed to converge with {\it probability one},
for {\it all} $\A$-tms $y$ that admits a $K$-measure.
If, accidently, a bad $R$ is generated such that
Algorithm~\ref{sdpalg:A-tkmp} fails to converge,
we can choose a different generic $R \in \Sig_{n,d}$.
Indeed, this never happened in our numerical experiments.
Theorem~\ref{cvgthm:asymp} guarantees that
we {\it almost always} succeed by choosing $R$ from $\Sig_{n,d}$.

\subsection{Finite convergence}

Now we characterize when Algorithm~\ref{sdpalg:A-tkmp} has finite convergence.
Denote $g_{m_2+1} := \rho-\|x\|_2^2$, then $g_B=(g_1,\ldots,g_{m_2}, g_{m_2+1})$.
For an index set $J=\{j_1,\ldots, j_r\}$, denote
$g_J :=(g_{j_1},\ldots,g_{j_r})$.

For a polynomial $f$,
denote by $f^{hom}$ the homogeneous part of $f$ with the highest degree.
If $p=(p_1,\ldots,p_r)$ is tuple,
denote $p^{hom} := (p_1^{hom},\ldots,p_r^{hom})$.
We denote by $Jac(p)|_{u}$ the Jacobian of $p$
evaluated at the point $u$.
The discriminant of $p^{hom}$, denoted as $\Delta(p^{hom})$,
is a polynomial in the coefficients of $p^{hom}$,
such that $\Delta(p^{hom})=0$ if and only if
$p^{hom}(x)=0$ has a nonzero solution $u\in \cpx^n$ with $\rank\, Jac(h^{hom})|_{u} < r$.
We refer to \cite[Section~3]{Nie-dis} for discriminants.

\begin{ass} \label{ass:gB-nsig}
For any $J \subseteq [m_2+1]$ with $V_{\re}(h, g_J) \ne \emptyset$,
the coefficients of $h$ and $g_B$ satisfy $\Delta(h^{hom}, g_J^{hom})\ne 0$.
\end{ass}

Assumption~\ref{ass:gB-nsig} holds generically,
because it requires $(h^{hom},g^{hom})$ to lie in an open dense set.
The finite convergence of Algorithm~\ref{sdpalg:A-tkmp} is as follows.

\begin{theorem} \label{thm:fin-cvg}
Let $K \subseteq B(0,\rho)$ be as in \reff{def:K}, and $y$ be an $\A$-tms.
Let $w^{*,k}$ be an optimizer of \reff{qmod:min<R,w>} with order $k$, if it exists.

\bit

\item [(i)] Suppose $\re[x]_{\A}$ is $K$-full.
If $y$ admits no $K$-measures,
then \reff{qmod:min<R,w>} is infeasible for all $k$ big enough.

\item [(ii)] Suppose $meas(y,K)\ne \emptyset$ and Assumption~\ref{ass:gB-nsig} holds.
If \reff{psd:max<p,y>} has a maximizer $p^*$ with
$R-p^* \in I_{2k_1}(h)+Q_{k_1}(g_B)$ for some $k_1$,
and if $R$ is generic in $\Sig_{n,d}$,
then $w^{*,k}|_{2t}$ is flat for some $t\geq d/2$, for $k$ big enough.

\item [(iii)] Suppose \reff{putsos:max<p,y>} achieves its maximum for $k$ big enough.
If the truncation $w^{*,k_2}|_{2t}$ is flat
for some $k_2$ and $t\geq d/2$, and if $R \in int(\Sig_{n,d})$,
then \reff{psd:max<p,y>} has a maximizer $p^*$ with
$R-p^* \in I_{2k_3}(h)+Q_{k_3}(g_B)$ for some $k_3$.

\eit
\end{theorem}
\begin{proof}
(i) If $\re[x]_{\A}$ is $K$-full and $y \not\in \mathscr{R}_{\A}(K)$,
by Corollary~\ref{cor:ctf-nomeas},
there exists $ p_1 \in \re[x]_{\A}$ such that
$
p_1|_K \geq 0, \, \langle p_1, y \rangle < 0.
$
By the $K$-fullness of $\re[x]_{\A}$, there exists $ p_2 \in \re[x]_{\A}$
with $p_2|_K>0$. So, for $\eps >0$ small, $\hat{p} := p_1 + \eps p_2$ satisfies
$
\hat{p} |_K > 0, \, \langle \hat{p} , y \rangle < 0.
$
Let $\eta_0>0$ be such that $(R+\eta_0 \hat{p})|_K>0$.
By Theorem~\ref{thm:PutThm}, both
$R+\eta_0 \hat{p}$ and $\hat{p}$ belong to
$\in I_{2t_1}(h) + Q_{t_1}(g_B)$, for some $t_1$.
Hence, for all $\eta > \eta_0$ and $k\geq t_1$,
$-\eta \hat{p}$ is feasible for \reff{putsos:max<p,y>}, because
\[
R+\eta \hat{p} =  R+\eta_0 \hat{p} + (\eta-\eta_0) \hat{p}
 \in I_{2k}(h) + Q_{k}(g_B).
\]
Note that $\langle -\eta \hat{p}, y\rangle \to +\infty$ as $\eta \to +\infty$.
So, the optimal value of \reff{putsos:max<p,y>} is $+\infty$.
By weak duality, its dual problem~\reff{qmod:min<R,w>} must be infeasible
for all $k \geq t_1$.

(ii) By Proposition~\ref{pro:sdprlx}(i), \reff{qmod:min<R,w>} and \reff{putsos:max<p,y>}
have the same optimal value, for all $k$, if $R \in int(\Sig_{n,d})$
(this is true if $R$ is generic in $\Sig_{n,d}$).
For all $k\geq k_1$, $p^*$ is feasible for \reff{putsos:max<p,y>}.
So, if $k\geq k_1$, then
\[
\langle R, w^{*,k} \rangle  = \langle p^*, w^{*,k} \rangle =\langle p^*, y \rangle,
\]
and
$
\langle R-p^*, w^{*,k} \rangle = 0.
$
Let $q := R-p^*$, then
\[
\langle q, w^{*,k} \rangle = 0
\quad \forall \, k\geq k_1.
\]
Clearly, $q$ is nonnegative on $K$.
Let $z^*$ be a minimizer of \reff{mom:min<R,z>} and $\mu \in meas(z,K)$.
Then, by Proposition~\ref{pr:optimizer} (i)
(note $R|_K>0$ if $R$ is generic in $\Sig_{n,d}$),
\[
0 = \langle R, z^*\rangle - \langle p^*, y \rangle
=  \langle q, z^* \rangle  = \int_K q \mathtt{d}\mu.
\]
Thus, $q$ vanishes on $\supp{\mu}$, and has a zero on $K$.

Consider the optimization problem:
\be \label{pop:min-q-K}
\min_x \quad q(x) \quad
s.t.  \quad h(x) = 0, \,  g_B(x) \geq 0.
\ee
The $k$-th order SOS relaxation for \reff{pop:min-q-K} is
\be \label{sos:minq}
\gamma_k := \max  \quad \gamma \quad s.t. \quad
q-\gamma \in  I_{2k}(h) + Q_k(g_B).
\ee
Its dual problem is
\be \label{mom:minq}
\min_{ w }  \quad \langle q, w \rangle \quad s.t. \quad
w \in \Phi_k(g_B) \cap E_k(h), w_0 = 1.
\ee
The minimum of $q$ over $K$ is $0$, and $\gamma_k = 0$ for all $k \geq k_1$.
Thus, the sequence $\{\gamma_k\}$ has finite convergence.
The SOS program \reff{sos:minq} achieves its optimal value for $k \geq k_1$,
because $q\in I_{2k_1}(h)+ Q_{k_1}(g_B)$.

Since $d>\deg(\A)$, $(R-p^*)^{hom} = R^{hom}$.
Under Assumption~\ref{ass:gB-nsig}, for every $J\subseteq [m_2+1]$
with $V_{\re}(h,g_J)\ne \emptyset$,
$\Delta(f, h^{hom}, g_J^{hom})$ is not constantly zero in $f$
(cf.~\cite[Theorem~3.2]{Nie-dis}).
So, if $R$ is generic in $\Sig_{n,d}$, then $\Delta(R^{hom}, h^{hom}, g_J^{hom}) \ne 0$
for all such $J$. By Proposition~\ref{pro:fin:kt:R-p} in the Appendix,
\reff{pop:min-q-K} has only finitely many critical points
and Assumption~2.1 in \cite{Nie-FT} for \reff{pop:min-q-K} is
satisfied\footnote{In \cite{Nie-FT}, polynomial optimization problems with only
inequality constraints were discussed. If there are equality constraints,
Assumption~2.1 in \cite{Nie-FT} can be naturally modified to include all
constraining equations, and Theorem~2.2 of \cite{Nie-FT} is still true,
with the same proof.}.

If $(w^{*,k})_{0}=0$, then $vec(1)^T M_k(w^{*,k}) vec(1) =0$,
and $M_k(w^{*,k}) vec(1) =0$ because $M_k(w^{*,k})\succeq 0$.
(Here $vec(p)$ denotes the coefficient vector of a polynomial $p$.)
This implies that $M_k(w^{*,k}) vec(x^\af)=0$ for all $|\af| \leq k-1$
(cf.~\cite[Lemma~5.7]{Lau}). For all $|\af|\leq 2k-2$,
we can write $\af = \bt + \eta$ with $|\bt|,|\eta| \leq k-1$,
and get
\[
(w^{*,k})_\af= vec(x^\bt)^T M_k(w^{*,k}) vec(x^\eta) =0.
\]
So, the truncation $w^{*,k}|_{2k-2}$ is flat.
%
%

If $(w^{*,k})_0>0$, we can scale $w^{*,k}$ so that $(w^{*,k})_0=1$.
Then $w^{*,k}$ is a minimizer of \reff{mom:minq}
because $\langle q, w^{*,k} \rangle = 0$ for all $k\geq k_1$.
By Theorem~2.2 of \cite{Nie-FT},
$w^{*,k}$ has a flat truncation $w^{*,k}|_{2t}$ if $k$ is big enough.
Indeed, $w^{*,k}|_{2k-2}$ is flat (cf. Remark~2.3 of \cite{Nie-FT}).
So, there is a flat truncation $w^{*,k}|_{2t}$
with $t\geq d/2$ if $k$ is big enough.

(iii) Suppose $w^{*,k_2}|_{2t}$ is flat and $2t\geq d$.
Let $R_{min}$ be the optimal value of \reff{mom:min<R,z>}, then
by Proposition~\ref{pro:sdprlx}(ii), $R_{min} \geq \langle R, w^{*,k_2} \rangle$.
On the other hand, the truncation $w^{*,k_2}|_{d}$
is feasible in \reff{mom:min<R,z>}, so
$R_{min} \leq \langle R, w^{*,k_2} \rangle$.
Thus, $R_{min} = \langle R, w^{*,k_2} \rangle$.
Indeed, we have $R_{min} = \langle R, w^{*,k} \rangle$ for all $k\geq k_2$.
By assumption, \reff{putsos:max<p,y>} has a maximizer $p^*$
at a big order, say, $k_3 \geq k_2$.
Then $R-p^* \in I_{2k_3}(h)+Q_{k_3}(g_B)$.
Clearly, $p^*$ is feasible for \reff{psd:max<p,y>} because $(R-p^*)|_K\geq 0$.
By Proposition~\ref{pro:sdprlx}(i), the optimal value of \reff{putsos:max<p,y>}
is also equal to $R_{min}$, if $R \in int(\Sig_{n,d})$.
So, $\langle p^*, y\rangle =\langle R, w^{*,k_3} \rangle = R_{min}$.
The optimal values of \reff{mom:min<R,z>} and \reff{psd:max<p,y>}
are equal, by Proposition~\ref{pr:optimizer}(i) if $R \in int(\Sig_{n,d})$.
Hence, $p^*$ is a maximizer of \reff{psd:max<p,y>}.
\end{proof}

\begin{remark} \label{rmk:K-ful}
By item (i) of Theorem~\ref{thm:fin-cvg},
if $\re[x]_{\A}$ is $K$-full and $y$ admits no $K$-measures,
then \reff{qmod:min<R,w>} is infeasible for some $k$, for {\it any} $R$
(we don't need $R \in \Sig_{n,d}$).
When $\re[x]_{\A}$ is not $K$-full and $y$ admits no $K$-measures,
it is not clear whether or not there exists $k$
such that \reff{qmod:min<R,w>} is infeasible.
This is because there does not exist a characterization like Theorem~\ref{thm2.2:NF03}
for the membership in $\mathscr{R}_{\A}(K)$ if $\re[x]_{\A}$ is not $K$-full,
to the best of the author's knowledge.
Theorem~\ref{thm2.2:NF03} and Corollary~\ref{cor:ctf-nomeas}
assume $K$-fullness of $\re[x]_{\A}$.
In many applications, $\re[x]_{\A}$ is often $K$-full.
On the other hand, if $y$ admits a $K$-measure,
no matter $\re[x]_{\A}$ is $K$-full or not, for a generic $R \in \Sig_{n,d}$,
Algorithm~\ref{sdpalg:A-tkmp} will find
a finitely atomic $K$-representing measure for $y$,
either asymptotically or in finitely many steps.
\end{remark}

\begin{remark}  \label{rmk:ficvg}
a) When $int(K)\ne\emptyset$, \reff{qmod:min<R,w>} has interior points,
and thus \reff{putsos:max<p,y>} achieves its optimal value,
for every order $k$ (cf. \cite{HN04,LasBok}).
b) By Theorem~\ref{thm:fin-cvg} (ii) and (iii), the condition $R-p^* \in I(h) + Q(g_B)$
is almost necessary and sufficient for finite convergence to occur,
modulo some general technical assumptions.
c) If a polynomial $f$ is nonnegative on $K$,
then $f \in I(h) + Q(g_B)$,
under some general conditions (cf. \cite{Nie-opcd}).
So, the condition $R-p^* \in I(h) + Q(g_B)$ is often satisfied.
Thus, it is very likely that Algorithm~\ref{sdpalg:A-tkmp} has finite convergence.
Indeed, the finite convergence occurred in all our numerical experiments.
\end{remark}

\section{Applications}\label{sec:appl}
\setcounter{equation}{0}

In this section, we show how Algorithm~\ref{sdpalg:A-tkmp} can be applied to
solve CP/SOEP-decomposition problems
and the standard truncated $K$-moment problems.

\subsection{Completely positive matrices}
\label{sbsec:cp}

Recall that a matrix $C \in \mc{S}^n$ is completely positive
if there exist $u_1,\ldots, u_r \in \re_+^n$ such that
\be \label{def:compos}
C = u_1u_1^T+\cdots+u_ru_r^T.
\ee
If \reff{def:compos} holds, we say $C$ is a CP-matrix.
The number $r$ is called the length of \reff{def:compos}.
The smallest such $r$ is called the {\it CP-rank} of $C$ (cf.~\cite{BerSM03}).
Let $\mathtt{Cp}(n)$ be the cone of $n\times n$ CP-matrices.
Clearly, $C\in \mathtt{Cp}(n)$ if and only if
$C=BB^T$ for a {\it nonnegative matrix} $B$ (i.e., every entry of $B$ is nonnegative).
So, every CP-matrix must be positive semidefinite, but typically not vice versa.
The dual cone of $\mathtt{Cp}(n)$ is $\mathtt{Co}(n)$,
the set of $n\times n$ copositive matrices
(a matrix $A \in \mc{S}^n$ is copositive if
$x^TAx \geq 0$ for all $x\in \re_+^n$).
Let
$
\Delta_n = \{ x\in \re_+^n:\, x_1+\cdots+x_n = 1 \}
$
be the standard simplex in $\re^n$.

Completely positive and copositive matrices have wide applications in optimization,
like approximating stability numbers (cf.~\cite{dKPas02})
or solving nonconvex quadratic programs (cf.~\cite{Bur09}).
Checking the membership in $\mathtt{Cp}(n)$ is NP-hard (cf. \cite{DicGij11}).
We refer to the survey \cite{Dur10} by D\"{u}r
and the book \cite{BerSM03} by Berman and Shaked-Monderer.
%
%
Recently, Lasserre~\cite{LasCO} proposes a convergent hierarchy of
outer approximations for $\mathtt{Co}(n)$,
the dual cone of $\mathtt{Cp}(n)$. It also gives a convergent hierarchy of
inner approximations for $\mathtt{Cp}(n)$.
Therefore, the membership in the interior of $\mathtt{Cp}(n)$
can be checked in finitely many step by the method in \cite{LasCO}.
%
%
When $C$ is acyclic or circular,
Dickinson and D\"{u}r \cite{DicDur12} showed that
checking complete positivity can be done in linear-time.
For general cases, Berman and Rothblum \cite{BerRot06} showed that
checking complete positivity and computing CP-ranks can be done
by using Renegar's algorithm on quantifier elimination \cite{Regr92}.
This is a symbolic algorithm. It typically runs in exponential time,
and is usually very expensive to implement. In the prior existing work,
there are no much efficient numerical methods for solving general
CP-decomposition problems, in the author's best knowledge.

Clearly, $C$ is a CP-matrix if and only if
\be \label{decom:com-pos}
C = \varrho_1 u_1 u_1^T + \cdots +  \varrho_r u_r u_r^T,
\ee
for some $u_1,\ldots, u_r  \in \Delta_n$,
$\varrho_1,\ldots, \varrho_r >0$.
Every symmetric matrix $C$ can be identified by the vector consisting of its entries
\[ \mathtt{c} = (C_{ij})_{ i \leq j }. \]
Let
$
\mc{Q}_n = \{ \af \in \N^n:\, |\af| = 2\}.
$
Then $\mathtt{c}$ is a $\mc{Q}_n$-tms, and \reff{decom:com-pos} is equivalent to
\[
\mathtt{c} =  \varrho_1 [u_1]_{\mc{Q}_n} + \cdots + \varrho_r [u_r]_{\mc{Q}_n}.
\]
Clearly, if $C\in \mathtt{Cp}(n)$, then
$\mathtt{c}$ admits a $\Delta_n$-measure.
Conversely, if $\mathtt{c}$ admits a $\Delta_n$-measure,
then $\mathtt{c}$ also admits a finitely atomic $\Delta_n$-measure like the above
(cf.~Proposition~\ref{pro:atms:extn}), and $C\in \mathtt{Cp}(n)$.
Thus, the CP-decomposition problem is essentially an $\A$-TKMP
with $\A =\mc{Q}_n, K=\Delta_n$.
The simplex $\Delta_n$ is in the form \reff{def:K},
with $h=(\mathbf{1}^Tx-1)$ and $g=(x_1, \ldots, x_n)$
($\mathbf{1}$ denotes the vector of all ones),
and $\Delta_n \subseteq B(0,1)$.
Note that $\re[x]_{\mc{Q}_n}$ is $\Delta_n$-full.

By the above, the CP-decomposition problem can be solved by Algorithm~\ref{sdpalg:A-tkmp}.
If $C \not\in \mathtt{Cp}(n)$, then Algorithm~\ref{sdpalg:A-tkmp}
will return a certificate for this (i.e., \reff{qmod:min<R,w>} is infeasible for some $k$),
by Theorem~\ref{thm:fin-cvg}(i).
If $C\in \mathtt{Cp}(n)$, then we can asymptotically get a flat extension of $\mathtt{c}$,
for almost all $R\in \Sig_{n,d}$ ($d>2$ is even), by Theorem~\ref{cvgthm:asymp}.
Moreover, we can likely get it in finitely many steps (cf.~Remark~\ref{rmk:ficvg}).
Indeed, finite convergence occurred in all our numerical experiments.
After getting a flat extension of $\mathtt{c}$,
we can get a $r$-atomic $\Delta_n$-representing measure for $\mathtt{c}$,
which then produces a CP-decomposition for $C$.

\begin{exm}
Consider the matrix:
\[
C =
\bbm
     6  &  4  &  1  &  2  &  2 \\
     4  &  6  &  0  &  1  &  3 \\
     1  &  0  &  3  &  1  &  2 \\
     2  &  1  &  1  &  2  &  1 \\
     2  &  3  &  2  &  1  &  5 \\
\ebm.
\]
We apply Algorithm~\ref{sdpalg:A-tkmp} to the corresponding $\mc{Q}_5$-tms
\[
\mathtt{c}:=(6,4,1,2,2,   6,0,1,3,   3,1,2,   2,1,  5).
\]
To get a decomposition of small length,
we run Algorithm~\ref{sdpalg:A-tkmp} for a couple of times
(cf.~Remark~\ref{rmk:min-spt}).
In each time, we got a CP-decomposition for $C$.
The smallest length we got is $5$,
which occurs in the factorization $C=BB^T$ with
\[
B =
\bbm
    1.0911  &  2.0836  &  0.0000  &  0.3148  &  0.6076  \\
    0.0000  &  1.6456  &  0.0000  &  1.8143  &  0.0000  \\
    0.1488  &  0.0000  &  1.0379  &  0.0000  &  1.3786  \\
    1.0797  &  0.3606  &  0.8087  &  0.2241  &  0.0000  \\
    0.5830  &  0.0000  &  0.0000  &  1.6535  &  1.3878
\ebm.
\]
The CP-rank of $C$ is $5$, because
$5=\rank\,C \leq \mbox{CP-rank}\,C \leq 5$.
\end{exm}

\begin{exm}
Consider the matrix (cf.~\cite[Example~2.9]{BerSM03}):
\[
C =
\bbm
     1  &  1  &  0  &  0  &  1 \\
     1  &  2  &  1  &  0  &  0 \\
     0  &  1  &  2  &  1  &  0 \\
     0  &  0  &  1  &  2  &  1 \\
     1  &  0  &  0  &  1  &  6 \\
\ebm.
\]
It is positive semidefinite, but not completely positive (cf.~\cite{BerSM03}).
We apply Algorithm~\ref{sdpalg:A-tkmp} to verify this fact.
It terminates at Step~2 with $k=2$,
because \reff{qmod:min<R,w>} is infeasible.
This confirms that $C$ is not a CP-matrix.
\end{exm}

\begin{exm} (random instances)
We apply Algorithm~\ref{sdpalg:A-tkmp} to randomly generated CP matrices.
If $C\in \mathtt{Cp}(n)$,
then $C$ admits a CP-decomposition \reff{decom:com-pos}
with length $r \leq \half n(n+1)$, by Carath\'{e}odory's Theorem.
Indeed, it can be slightly sharpened to
$r \leq \half \rank\,C(\rank\,C+1)-1$, if $\rank\,C \geq 2$
(cf.~\cite{BarBer03,LKF04}).
Clearly, we always have $r \geq \rank (C)$.
So, if $C\in \mathtt{Cp}(n)$ and $C$ has full rank,
then $n \leq r \leq cp(n):=\half n(n+1)-1$, for $n>1$.
For $n=2,3,\ldots,8$, we generate $50$ instances,
except for $n=8$ (only $20$ instances are generated).
For each instance, generate $N:=\half n(n+1)$ points randomly from $\Delta_n$,
say, $u_1, \ldots, u_N$, and let
$C = c_1 u_1u_1^T+\cdots+ c_N u_N u_N^T$ with $c_i>0$ random.
For each $C$, we apply Algorithm~\ref{sdpalg:A-tkmp} ten times
and let $r$ be the smallest length that is obtained.
Algorithm~\ref{sdpalg:A-tkmp}
is able to get a CP-decomposition for all generated $C$.
The obtained values of $r$ are listed in the table:
\bcen
\btab{|c|c|c|c|c|c|c|c|} \hline
$n$       & 2  & 3 & 4 & 5 & 6 & 7 & 8 \\ \hline
$cp(n)$ & 2  & 5 & 9 & 14 & 20 & 27 & 35 \\ \hline
$r$     & 2  & 3 & 4 & 5,6 & 6,7,8  & 8,9,10 & 11,12,13,14,15 \\ \hline
\etab.
\ecen
They are equal or close to the lower bound $n$
(because $\rank\,C=n$ for generated $C$),
and is much less than the upper bound $cp(n)$ for $n\geq 4$.
\end{exm}

\subsection{Sum of even powers (SOEP) of real linear forms}
\label{sbsec:soep}

Recall that a form $f$ of an even degree $m$ is SOEP if
for some real linear forms $L_1,\ldots,L_r$
\be   \label{decom:soep}
f = L_1^{m}+\cdots+L_r^{m}.
\ee
Let $Q_{n,m}$ denote the set of all SOEP forms in $n$ variables and of degree $m$.
Reznick proved that $Q_{n,m}$ is a convex cone with nonempty interior
and its dual cone is the set of nonnegative forms in $n$ variables and of degree $m$.
We refer to Reznick \cite{Rez92} for SOEP forms.
The number of sums, $r$, is called the length of \reff{decom:soep}.
The minimum $r$ for which \reff{decom:soep} holds
is called the {\it width} of $f$, and is denoted as $w(f)$ (cf.~\cite{Rez92}).
The decomposition~\reff{decom:soep} is called {\it minimum} if $r=w(f)$.
SOEP decompositions naturally have wide and interesting applications,
like in Waring's problems, quadrature problems, sphere designs \cite{Rez92}.
It is typically quite difficult to check whether a form is SOEP or not.
As shown by Reznick \cite{Rez92}, when $m\geq 4$,
a rational form $f \in Q_{n,m}$
may not have a decomposition \reff{decom:soep} with all $L_i$ rational.
Therefore, numerical methods are preferable in applications.
In the prior existing work, there are no much efficient numerical methods
for solving SOEP decomposition problems,
in the author's best knowledge.

Let $\mathbb{H}_{m}^n = \{\af \in \N^n:\, |\af|=m\}$.
We can write a form $f$ of degree $m$ as
\[
f = \sum_{ \af \in \mathbb{H}_{m}^n }  \binom{m}{\af}
\check{f}_\af x^\af.
\]
Denote $\check{f}:=(\check{f}_\af)_{ \af \in \mathbb{H}_{m}^n } $.
So, $f$ can be identified by the $\mathbb{H}_{m}^n$-tms $\check{f}$.
If $f$ is SOEP and \reff{decom:soep} holds,
then we can write each $L_i = \sqrt[m]{c_i}(u_i^Tx)$ with $c_i>0$ and
\[
u_i \in \mathbb{S}^{n-1}_+:=\{x\in \mathbb{S}^{n-1}: \mathbf{1}^Tx \geq 0\}.
\]
Thus, we get
\[
f = \sum_{i=1}^r \sum_{ \af \in \mathbb{H}_{m}^n }  \binom{m}{\af} c_i u_i^\af x^\af.
\]
The above is equivalent to the decomposition:
\be  \label{check(f):htmp}
\check{f} =  c_1 [u_1]_{\mathbb{H}_{m}^n} +\cdots+ c_r [u_r]_{\mathbb{H}_{m}^n}.
\ee
Clearly, if $f$ is SOEP, then $\check{f}$ admits a $\mathbb{S}^{n-1}_+$-measure.
Conversely, if $\check{f}$ admits an $\mathbb{S}^{n-1}_+$-measure,
then $\check{f}$ also admits a finitely atomic $\mathbb{S}^{n-1}_+$-measure
(cf. Proposition~\ref{pro:atms:extn}), and so $f$ is SOEP.
Hence, checking $f \in Q_{n,m}$
is equivalent to determining whether the $\mathbb{H}_{m}^n$-tms $\check{f}$
admits a $\mathbb{S}^{n-1}_+$-measure or not.
The latter is a $\A$-TKMP with $\A=\mathbb{H}_{m}^n$ and $K=\mathbb{S}^{n-1}_+$.
Note that $\re[x]_{ \mathbb{H}_{m}^n }$ is $\mathbb{S}^{n-1}_+$-full.

SOEP-decomposition problems can be solved by Algorithm~\ref{sdpalg:A-tkmp}.
The set $\mathbb{S}^{n-1}_+$ is as in \reff{def:K}
with $h=(\|x\|_2^2-1)$ and $g=(\mathbf{1}^Tx)$.
Clearly, $\mathbb{S}^{n-1}_+ \subseteq B(0,1)$.
In \reff{qmod:min<R,w>}, the tuple $g_B$ can be replaced by $g$,
because $\Psi_k(g) \cap E_k(h)=\Psi_k(g_B) \cap E_k(h)$.
If $f$ is not SOEP, then $\check{f}$ admits no $\mathbb{S}^{n-1}_+$-measures,
and Algorithm~\ref{sdpalg:A-tkmp} can give a certificate for this
(i.e.,\reff{qmod:min<R,w>} is infeasible for some order $k$), by Theorem~\ref{thm:fin-cvg}(i).
If $f$ is SOEP, then we can asymptotically get
a flat extension of $\check{f}$, for almost all
$R \in \Sig_{n,d}$ ($d>m$ is even), by Theorem~\ref{cvgthm:asymp}.
Moreover, we can likely get it in finitely many steps (cf. Remark~\ref{rmk:ficvg}).
Indeed, this occurred in all our numerical experiments.
Once a flat extension of $\check{f}$ is obtained,
we can easily get an SOEP-decomposition for $f$
from a finitely atomic representing measure for $\check{f}$.

\begin{exm}
Consider the sextic form
\[
q_{\lmd} := (x_1^2+x_2^2+x_3^2)^3 - \lmd (x_1^6+x_2^6+x_3^6).
\]
It is SOEP if and only if $\lmd \leq 2/3$ (cf.~\cite[p.~146]{Rez92}).
For $\lmd = 2/3$, by running Algorithm~\ref{sdpalg:A-tkmp} a few times,
we got an SOEP decomposition of length $10$ for $q_{2/3}$:
{\small
\[
\baray{c}
\frac{1}{15} \left(
(x_1+x_2)^6+(x_2+x_3)^6+(x_1+x_3)^6 + (x_1-x_2)^6 +
 (x_1-x_3)^6+ (x_2-x_3)^6 \right)  + \\
\frac{1}{60} \left(
(x_1+x_2+x_3)^6 +
(-x_1+x_2+x_3)^6 + (x_1-x_2+x_3)^6  + (x_1+x_2-x_3)^6
\right).
\earay
\]}
\noindent
For $\lmd = 1/3$, we can get an SOEP-decomposition of length $11$ for $q_{1/3}$,
by the same way. The lengths $10$ and $11$ are the smallest ones
that we can get for $q_{2/3}$ and $q_{1/3}$ respectively
(cf. Remark~\ref{rmk:ficvg}).
When $\lmd = 1$, \reff{qmod:min<R,w>} is infeasible for $k=4$,
and Algorithm~\ref{sdpalg:A-tkmp} terminates at Step~2.
This confirms $q_{1} \not\in Q_{3,6}$.
\end{exm}

\begin{exm} (random instances)
We apply Algorithm~\ref{sdpalg:A-tkmp} to
randomly generated SOEP forms. Let $m$ be an even degree.
If $f \in Q_{n,m}$, then its width $w(f) \leq N:=\binom{n+m-1}{m}$, by Carath\'{e}odory's Theorem.
If $f \in int(Q_{n,m})$, then its width $w(f) \geq N_0:=\binom{n+m/2-1}{m/2}$
(cf.~\cite[Theorem~3.14(iv)]{Rez92}).
If $n=2$ or $(n,m)=(3,4)$, then $w(f) \leq N_0$ (cf.~\cite[Theorem~4.6]{Rez92}).
So, for the above range of $(n,m)$, we know the generic width is $N_0$.
For other values of $(n,m)$, if $f$ is generic inside $Q_{n,m}$, then $N_0 \leq w(f) \leq N$.
We consider $(n,m)$ from the table:
\bcen
\btab{|c|c|c|c|c|c|c|c|c|} \hline
$(n,m)$      & (2,4) & (2,6) & (2,8) & (2,10) &  (3,4) &  (4,4)    &  (3,6)   &  (3,8) \\ \hline
{\tt gwidth}  &   3   &   4   &   5   &    6   &    6   &  [10,35]  & [10,28]  &  [15,45] \\ \hline
   $r$        &   3   &   4   &   5   &    6   &   6,7  &  12,13,14 &   11,12  &  17,18,19 \\ \hline
\etab.
\ecen
In the above, {\tt gwidth} is $N_0$ if $n=2$ or $(n,m)=(3,4)$,
and is the range $[N_0,N]$ for other cases.
For each pair $(n,m)$ from the above table, we generate $50$ instances.
In each instance, generate points $u_1, \ldots, u_{N}$ randomly  from
$\mathbb{S}^{n-1}$, and let
$f = c_1 (u_1^Tx)^m+\cdots+  c_N (u_{N}^Tx)^m$ with $c_i>0$ random.
For each generated $f$, we run Algorithm~\ref{sdpalg:A-tkmp}
for ten times, and choose $r$ to be the smallest length of
the obtained SOEP-decompositions.
For all generated $f$, we got an SOEP-decomposition,
and the values of obtained lengths $r$ are listed in the above table.
We can see that $r$ is equal or close to the minimum.
\end{exm}

\subsection{Standard truncated $K$-moment problems}

When $\A= \N_m^n$, the $\A$-TKMP is specialized to the standard
truncated $K$-moment problem (TKMP),
which was originally studied by Curto and Fialkow \cite{CF91,CF96,CF05,CF09}.
Algorithm~\ref{sdpalg:A-tkmp} can be naturally applied to solve TKMPs.
The set $\re[x]_m$ is $K$-full, for any set $K$.
We are interested in the case that $K$ is compact.
If a tms $y \in \re^{ \N_m^n }$ admits no $K$-measures, Algorithm~\ref{sdpalg:A-tkmp}
can return a certificate for the nonexistence of representing measures;
if $y$ admits a $K$-measure, we can asymptotically get a flat extension of $y$,
for almost all $R \in \Sig_{n,d}$ ($d>m$ is even), by Theorems~\ref{cvgthm:asymp};
moreover, we can likely get it in finitely many steps
(cf. Theorem~\ref{thm:fin-cvg} and Remark~\ref{rmk:ficvg}).
In the author's best knowledge, Algorithm~\ref{sdpalg:A-tkmp} is
the first numerical algorithm that can
solve general TKMPs with a compact semialgebraic set $K$.

\begin{exm}
Consider the following tms in $\re^{ \N_6^2 }$:
\bcen
$
\baray{c}
(1,
0, 0,
1/3, 0, 1/3,
0, 0, 0, 0,
1/5, 0, 1/9, 0, 1/5, \\
0, 0, 0, 0, 0, 0,
1/7, 0, 1/15, 0, 1/15, 0, 1/7).
\earay
$
\ecen
Its moments are listed in the graded lexicographical ordering.
This tms admits a measure supported on $[-1,1]^2$,
because its $\af$-th moment is the mean value of $x^\af$ on $[-1,1]^2$.
We apply Algorithm~\ref{sdpalg:A-tkmp} to this tms.
In each time of running, we got a $r$-atomic representing measure
supported in $[-1,1]^2$. After a repeated running,
the smallest $r$ we got is $10$ (cf. Remark~\ref{rmk:min-spt}),
which occurs in the representing measure $\Sig_{i=1}^{10} c_i \dt(u_i)$
with $u_i$ and $c_i$ given as:
\bcen
\btab{|cc|cc|}  \hline
$u_i$ &   $c_i$  &     $u_i$      &  $c_i$   \\  \hline
(-0.7983,  -0.9666) &  0.0318   &  (-0.8710,   -0.2228)  &  0.0947   \\  \hline
(-0.2155,  -0.6211)  & 0.1698   &  (-0.9175,  0.8643)  & 0.0328    \\  \hline
(0.5833,  -0.8876) &  0.0729   &  (-0.4834,   0.4714) &  0.1662   \\  \hline
(0.3541,  0.0676)  &  0.2054  & (0.9269,   -0.3819) &  0.0672    \\  \hline
(0.0841,  0.9294)  &  0.0717  & (0.8153,   0.6919) &  0.0874   \\  \hline
\etab.
\ecen
\end{exm}

\bigskip \noindent
{\bf Acknowledgement}\,
The author was partially supported by the NSF grants
DMS-0844775, DMS-1417985.
He would like very much to thank Raman Sanyal and Bernd Sturmfels
for comments on generic linear optimization over compact convex sets.

\appendix
\section{Generic Finiteness of critical varieties}
\setcounter{equation}{0}

Consider the polynomial optimization problem
\be \label{min:f:K}
\left\{\baray{rl}
\min & f(x)  \\
s.t.  & h_i(x) = 0\,(i \in [m_1]), \,
g_j(x) \geq 0\,( j\in [m_2+1]).
\earay \right.
\ee
For each $J  \subseteq \{1,\ldots, m_2+1\}$, denote
\[
\mc{V}_J := \{ x\in \cpx^n: \, h(x)=0,  g_J(x)=0,\,
\rank \, Jac(f,h,g_J)|_x \leq m_1+|J| \}.
\]
The set of critical points of \reff{min:f:K} with the active set $J$
is contained in $\mc{V}_J$, which is called a critical variety of \reff{min:f:K}.
We show that if the coefficients of $f^{hom}, h^{hom}, g_J^{hom}$
satisfy some discriminantal inequalities, then $\mc{V}_J$ is finite.
We refer to \cite[Section~3]{Nie-dis}) for the
definition of discriminants $\Delta$.

\begin{pro} \label{pro:fin:kt:R-p}
Let $f,h_i\,(i\in [m_1]), g_j \,(j\in [m_2+1]) \in \re[x]$,
and $\mc{V}_J$ be defined as above.
For any $J \subseteq [m_2+1]$, if
\[
\Delta(f^{hom}, h^{hom}, g_J^{hom}) \ne 0, \quad
\Delta(h^{hom}, g_J^{hom}) \ne 0,
\]
then $\mc{V}_J$ is finite.
\end{pro}
\begin{proof}
Denote $\tilde{x}:=(x_0,x_1,\ldots,x_n)$
and by $\tilde{p}$ the homogenization of a singleton or tuple of polynomials $p$.
Let $\mc{U}_J$ be the projective variety in $\P^n$ (cf. \cite{Har}) defined as
\be \label{ktjac:hmg:x}
\rank\, Jac(\tilde{f}, \tilde{h},\widetilde{g_J})|_x \leq m_1 + |J|,  \quad
\widetilde{h}(\tilde{x}) =\widetilde{g_J}(\tilde{x})  =0.
\ee
Clearly, if $u\in \mc{V}_J$, then  $(1,u) \in \mc{U}_J$.
Suppose otherwise $\mc{V}_J$ is infinite, then $\mc{U}_J$ is positively dimensional.
By Bezout's Theorem (cf. \cite{Har}),
$\mc{U}_J$ must intersect the hyperplane $x_0=0$ in $\P^n$,
i.e., \reff{ktjac:hmg:x} has a solution like $(0,v)$ with $0\ne v \in \cpx^n$.
So, $v$ is a solution to the homogeneous polynomial system
\be \label{Jac:fhgJ:x}
\rank\, Jac(f^{hom}, h^{hom},g_J^{hom})|_x\leq m_1 + |J|,   \quad
 h^{hom}(x) =g_J^{hom}(x)=0.
\ee
Since $\Delta(h^{hom}, g_J^{hom}) \ne 0$,
$\rank\, Jac(h^{hom},g_J^{hom})|_x = m_1 + |J|$
for all $0 \ne x \in V_{\cpx}(h^{hom},g_J^{hom})$.
The rank condition in \reff{Jac:fhgJ:x} implies that $f^{hom}(v)=0$
(cf.~\cite[Section~3]{Nie-dis}). Hence, $v$ is a nonzero singular solution to
\[
f^{hom}(x) = h^{hom}(x) =g_J^{hom}(x)=0,
\]
which contradicts $\Delta(f^{hom}, h^{hom}, g_J^{hom}) \ne 0$
(cf.~\cite{Nie-dis}). So, $\mc{V}_J$ must be finite.
\end{proof}

%
%


\begin{thebibliography}{99}

\bibitem{BT}
C. Bayer and J. Teichmann.
The proof of Tchakaloff's Theorem,
{\it Proc. Amer. Math. Soc.}, 134(2006), pp. 3035-3040.


\bibitem{BTN}
A.~Ben-Tal and A.~Nemirovski.
{\it Lectures on Modern Convex Optimization: Analysis, Algorithms, and Engineering
Applications}. MPS-SIAM Series on Optimization, SIAM, Philadelphia, 2001.


\bibitem{BarBer03}
F.~Barioli, A.~Berman.
The maximal CP-rank of rank $k$ completely positive matrices.
{\it Linear Algebra Appl.} 363(2003), pp. 17-33.



\bibitem{BerRot06}
A.~Berman and U.~Rothblum.
A note on the computation of the CP-rank.
{\it Linear Algebra and its Applications} \, 419 (2006), pp. 1-7.


\bibitem{BerSM03}
A Berman and N. Shaked-Monderer.
{\it Completely Positive Matrices},
World Scientific, 2003.

%
%

\bibitem{Bur09}
S. Burer.
On the copositive representation of binary and continuous
nonconvex quadratic programs.
{\it Mathematical Programming}, Ser. A, 120(2009), pp. 479-495.


\bibitem{Conway}
J.~B.~Conway. {\it A course in Functional Analysis}.
Springer-Verlag, 1990 $2^{nd}$ edition.


\bibitem{CF91}
R. Curto and L. Fialkow,
Recursivenss, positivity, and truncated moment problems.
{\it Houston J. Math.} 17(1991), pp. 603-635.


\bibitem{CF96}
R. Curto and L. Fialkow.
Solution of the truncated complex moment problem for flat data.
{\it Memoirs of the American Mathematical Society}, 119(1996), No. 568,
Amer. Math. Soc., Providence, RI, 1996.


\bibitem{CF982}
R. Curto and L. Fialkow.
Flat extensions of positive moment matrices: Relations in
analytic or conjugate terms.
{\it Operator Th.: Adv. Appl.} 104(1998), pp. 59-82.


%
%

%
%


\bibitem{CF05}
R. Curto and L. Fialkow.
Truncated K-moment problems in several variables.
{\it Journal of Operator Theory},  54(2005), pp. 189-226.



\bibitem{CF09}
R. Curto and L. Fialkow.
An analogue of the Riesz-Haviland Theorem for the truncated moment problem,
{\it J. Functional Analysis}, 225(2008), 2709-2731.

%
%

%


\bibitem{dKPas02}
E.~de Klerk and D.V.~Pasechnik.
Approximating of the stability number of a graph via copositive programming.
{\it SIAM Journal on Optimization}, 12(4), 875-892.


%
%


\bibitem{DicGij11}
P.J.~Dickinson and L.~Gijben.
On the computational complexity of membership problems
for the completely positive cone and its dual.
{\it Computational Optimization and Applications}\, 57 (2014), No.~2, 403--415.


\bibitem{DicDur12}
P.J.~Dickinson and M.~D\"{u}r.
Linear-time complete positivity detection and decomposition of sparse matrices.
{\it SIAM Journal On Matrix Analysis and Applications},
33 (2012), No.~3, 701-720.


\bibitem{Dur10}
M.~D\"{u}r.
Copositive Programming - a Survey.
In: M. Diehl, F. Glineur, E. Jarlebring, W. Michiels (Eds.),
{\it Recent Advances in Optimization and its Applications in Engineering},
Springer 2010, pp. 3-20.


\bibitem{FiNi1}
L. Fialkow and J. Nie.
Positivity of Riesz functionals and solutions of
quadratic and quartic moment problems,
{\it J. Functional Analysis},
258(2010), no. 1, pp. 328--356.


\bibitem{FiNi2012}
L. Fialkow and J. Nie.
The truncated moment problem via homogenization
and flat extensions.
{\it J. Functional Analysis}
263(2012), no. 6, pp. 1682--1700.

%
%


\bibitem{Har}
J. Harris.
{\it Algebraic Geometry, A First Course}. Springer Verlag, 1992.



\bibitem{HN04}
J.W.~Helton and J. Nie.
A semidefinite approach for truncated K-moment problems.
{\it Foundations of Computational Mathematics},
Vol.~12, No.~6, pp. 851-881, 2012.


\bibitem{HenLas05}
D.~Henrion and J.~Lasserre.
Detecting global optimality and extracting solutions in GloptiPoly.
{\it Positive polynomials in control},  293-310,
Lecture Notes in Control and Inform. Sci., 312,
Springer, Berlin, 2005.


\bibitem{GloPol3}
D.~Henrion, J.~Lasserre and J.~Loefberg.
{\it GloptiPoly 3: moments, optimization and semidefinite programming}.
{\it Optimization Methods and Software} 24(2009), no. 4-5, pp. 761-779.


\bibitem{Las01}
J. B. Lasserre. Global optimization with polynomials and the problem of moments.
{\it SIAM J. Optim.} 11(2001), no.3, pp. 796-817.

%
%


\bibitem{LasSpar}
J.~B.~Lasserre.
Convergent SDP-relaxations in polynomial optimization with sparsity.
{\it SIAM J. Optim. 17,} \,  pp. 822--843, 2006.


\bibitem{Las08}
J.B.~Lasserre J.B.
A semidefinite programming approach to the generalized problem of moments.
{\it Mathematical  Programming}\, 112 (2008),  pp. 65--92.



\bibitem{LasBok}
J.~B.~Lasserre.
{\it Moments, Positive Polynomials and Their Applications},
Imperial College Press, 2009.


\bibitem{LasCO}
J.~B.~Lasserre.
New approximations for the cone of copositive matrices and its dual.
{\it Mathematical  Programming}, to appear.


\bibitem{Lau05}
M. Laurent.
Revisiting two theorems of Curto and Fialkow on moment matrices.
{\it Proceedings of the American Mathematical Society} 133(2005), no. 10,
pp. 2965--2976.


\bibitem{Lau}
M. Laurent.
Sums of squares, moment matrices and optimization over polynomials,
Emerging Applications of Algebraic Geometry, Vol. 149 of
IMA Volumes in Mathematics and its Applications, M. Putinar and
S. Sullivant (eds), Springer, pages 157-270, 2009.



\bibitem{LauMou09}
M. Laurent and B. Mourrain.
A generalized flat extension theorem for moment matrices.
{\it Archiv der Mathematik}, \, 93(1), pp. 87--98,  2009.


\bibitem{LKF04}
Y.~Li, A.~Kummert, A.~Frommer.
A linear programming based analysis of the CP-rank
of completely positive matrices.
{\it Int. J. Appl. Math. Compu. Sci.} 14 (2004), pp. 25-31.



\bibitem{MarBk}
M. Marshall.
{\it Positive Polynomials and Sums of Squares}.
Mathematical Surveys and Monographs, 146.
American Mathematical Society, Providence, RI, 2008.


\bibitem{Nie-dis}
J.~Nie. Discriminants and nonnegative polynomials.
{\it Journal of Symbolic Computation}
47(2012), no.~2, pp. 167-191.


\bibitem{Nie-FT}
J.~Nie.
Certifying convergence of Lasserre's hierarchy via flat truncation.
{\it Mathematical Programming}, Ser. A, Vol~142, No. 1-2, pp. 485-510, 2013.


\bibitem{Nie-opcd}
J.~Nie.
Optimality conditions and finite convergence of Lasserre's hierarchy.
{\it Mathematical Programming}, Ser. A, to appear.

%
%


%
%



\bibitem{Put}
M.~Putinar.
Positive polynomials on compact semi-algebraic sets,
{\it Ind. Univ. Math. J.} 42(1993), pp. 969-984.




\bibitem{Regr92}
J.~Renegar.
On the computational complexity and geometry of
the first-order theory of the reals.
Parts I, II, III.
{\it J. Symbolic Comput.} 13 (1992), no. 3, 255-352.


\bibitem{Rez00}
B.~Reznick. Some concrete aspects of Hilbert's $17$th problem.
In {\it Contemp. Math.}, Vol.~253, pp. 251-272. American
Mathematical Society, 2000.


\bibitem{Rez92}
B.~Reznick.
Sums of even powers of real linear forms.
{\it Mem. Amer. Math. Soc.}, Vol.~96, No.~463, 1992.



\bibitem{Smg}
K. Schm\"{u}dgen.
The K-moment problem for compact semialgebraic sets.
{\it Math. Ann.}  289 (1991), 203–206.


\bibitem{Sneid}
R.~Schneider.
{\it Convex bodies: the Brunn-Minkowski theory.}
Encyclopedia of Mathematics and its Applications, Vol.~44.
Cambridge University Press, Cambridge, 1993.


\bibitem{sedumi}
J.F. Sturm. SeDuMi 1.02: A MATLAB toolbox for optimization over
symmetric cones. {\it Optimization Methods and Software}, 11 \& 12
(1999), pp. 625--653.


\bibitem{Stu02}
B. Sturmfels. {\it Solving systems of polynomial equations}.
CBMS Regional Conference Series
in Mathematics, 97. American Mathematical Society, Providence, RI, 2002.


\bibitem{Tch}
V.~ Tchakaloff. Formules de cubatures m$\acute{e}$canique
$\grave{a}$ coefficients non n$\acute{e}$gatifs.
{\it Bull. Sci. Math. (2)} 81(1957), pp. 123-134.


\bibitem{Todd}
M. Todd. Semidefinite optimization.
{\it Acta Numerica} 10(2001), pp. 515--560.



\end{thebibliography}
\end{document}